\documentclass{amsart}

\usepackage[a4paper]{geometry}
\usepackage[T1]{fontenc}
\usepackage{amssymb, amsmath,mathtools}
\usepackage{mathrsfs}
\usepackage{amscd}
\usepackage{verbatim}
\usepackage{stmaryrd}
\usepackage[dvipsnames]{xcolor}
\usepackage{enumerate}
\usepackage{xfrac}
\usepackage{bbm}
\usepackage{graphicx}
\usepackage{bbm}
\usepackage{enumitem}
\usepackage{accents}

\usepackage{mathtools}
\mathtoolsset{showonlyrefs}

\usepackage[colorlinks,linkcolor={blue},citecolor={blue},urlcolor={purple},]{hyperref}

\usepackage{tabularx}
\newcolumntype{L}{>{\arraybackslash}X}
\usepackage{multirow}

\theoremstyle{plain}
\newtheorem{theorem}{Theorem}[section]
\theoremstyle{remark}
\newtheorem{remark}[theorem]{Remark}

\newtheorem{convention}[theorem]{Convention}
\theoremstyle{plain}
\newtheorem{corollary}[theorem]{Corollary}
\newtheorem{lemma}[theorem]{Lemma}
\newtheorem{proposition}[theorem]{Proposition}
\newtheorem{definition}[theorem]{Definition}

\newtheorem{assumption}[theorem]{Assumption}

\numberwithin{equation}{section}

\def\N{{\mathbb N}}
\def\Z{{\mathbb Z}}

\def\R{{\mathbb R}}

\def\one{{ \mathbbm{1} }}

\renewcommand{\P}{\mathbb{P}}
\newcommand{\E}{\mathbb{E}}
\newcommand{\T}{\mathbb{T}}

\newcommand{\cF}{\mathcal{F}}

\newcommand{\norm}[1]{\left\lVert #1 \right\rVert}

\def\eps{\varepsilon}
\def\<{{\langle}}
\def\>{{\rangle}}

\newcommand{\dd}{\,\mathrm{d}}

\renewcommand{\leq}{\leqslant}
\renewcommand{\geq}{\geqslant}

\DeclareMathOperator*{\dist}{\textup{d}}

\DeclareMathOperator*{\supp}{\textup{supp}}

\DeclareMathOperator{\even}{\mathrm{even}}

\allowdisplaybreaks

\title[Non-Selection of Lagrangian Trajectories in the Zero-Noise Limit]{Non-Selection of Lagrangian Trajectories in the Zero-Noise Limit for a Class of Stochastic Regularizations}
\author[L. Galeati]{Lucio Galeati$^\star$}
\author[F. Giovagnini]{Filippo Giovagnini$^\dagger$}
\author[M. Sorella]{Massimo Sorella$^\ddagger$}
\thanks{$^\star$University of L'Aquila. lucio.galeati@univaq.it}
\thanks{$^\dagger$Imperial College London. f.giovagnini23@imperial.ac.uk }
\thanks{$^\ddagger$Imperial College London. m.sorella@imperial.ac.uk}

\begin{document}

\begin{abstract}
We prove the lack of selection in the zero-noise limit for solutions to SDEs driven by a divergence-free, H\"older continuous vector field with exponent $\alpha\in(0,1)$, arbitrarily close to $1$ but fixed. The result applies to a broad class of regularizing additive noises, including fractional Brownian motion and stable Lévy processes.
The proof combines pathwise Lagrangian arguments, based on the analysis of the deterministic flows associated to mixing velocity fields, with probabilistic estimates coming from the stochastic sewing lemma. This allows to show that lack of selection happens simultaneously on a large set of initial data, whose complement has arbitrarily small Lebesgue measure.\\[1ex]
\end{abstract}

\maketitle

\setcounter{tocdepth}{1}
\tableofcontents

\section{Introduction}

A classical question in ODE theory is to understand under which conditions on the vector field $u$
one can ensure the existence of a well-defined dynamics for the ODE problem (set on $\R^d$ or the torus $\T^d$)
\begin{align*}
    \dot x_t = u(t,x_t).
\end{align*}
A standard assumption is to require Lipschitz continuity of $u$ in the $x$-variable, possibly with some integrability in time. Slightly more refined conditions involve e.g. one-sided Lipschitz bounds~\cite{LioSee2024} or an Osgood modulus of continuity~\cite{AB2008}.
On the other hand, it is well-known that the ODE can admit multiple solutions as soon as $u$ is only $\alpha$-H\"older continuous, for any $\alpha<1$.

A more nuanced way to phrase the problem is then to understand whether one can provide interesting \emph{selection criteria} to distinguish some more relevant solutions among the possibly infinitely many ones.
A probabilistic way to approach this problem is by studying the \emph{vanishing noise limit} of the SDEs
\begin{equation}\label{eq:intro_SDE_v0}
    \dd X^\kappa_t = u(t,X^\kappa_t) \dd t + \kappa \dd W_t,
\end{equation}
where $W$ is an additive stochastic perturbation, as $\kappa\to 0^+$.
Roughly speaking, solutions obtained in this way should be stable under noisy perturbations and therefore more likely to be observed in realistic models; it is also of interest to understand whether a \emph{true superposition solution} can arise in the limit as $\kappa\to 0^+$, namely whether the laws of $X^\kappa$ do not concentrate on a single point, but rather converge to a genuine probability distribution supported on the solutions to the ODE.
For a deeper discussion on vanishing  noise limits, we refer to~\cite{Flandoli2015} and in particular Section~1.4 therein; let us point out in passing that this procedure bears strong similarities with vanishing viscosity solutions to PDE problems, see e.g. \cite{BB2005,CIL1992}.

In dimension $d=1$, the vanishing noise problem of the deterministic flow was successfully understood by Bafico and Baldi~\cite{BafBal1982}, for autonomous velocity fields $u$ with isolated zeros, displaying non-uniqueness scenarios in the style of Peano paintbrush phenomena.
In this case, when $W$ is sampled as Brownian motion, the limit of $Law(X^\kappa)$ can be completely characterized and shown to be supported only on extremal solutions to the ODE problem. This result has been revisited, e.g. in~\cite{DelFla2014,Trevisan2013}.
However, extensions of these results to dimensions $d\geq 2$ have proven to be considerably more challenging.
There are relatively few results available, see~\cite{DelMau2019,PP2018,KulPil2020}, concerning drifts $u$ with strong structural assumptions which are hard to verify in practice, typically Lipschitz regular away from a lower-dimensional set (a point or a hyperplane).

In this paper we show that, for a large class of stochastic perturbations, in dimension $d\geq 2$, one cannot expect the existence of a unique vanishing noise limit, under mere H\"older regularity and incompressibility assumptions on the velocity field $u$.
To state our main result, let us make our setting more precise.
We study the stochastic differential equation in integral form
\begin{equation}\label{eq:intro_SDE}
    X_t^\kappa
    =
    x+\int_0^t u(r,X_r^\kappa)\dd r+\kappa W_t ,
\end{equation}
where \(u:[0,1]\times\T^d\to\R^d\), \(d\geq 2\), is divergence-free (in the sense of distributions) and H\"older continuous, \(\kappa> 0\) is the noise intensity, and \((W_t)_{t\geq 0}\) is a stochastic process satisfying the following:
 
\begin{assumption}\label{ass:noise}
    The $\R^2$-valued stochastic process $W$ satisfies either of the following:
    \begin{itemize}
        \item $W$ is a fractional Brownian motion (fBm) with Hurst parameter $H\in (0,+\infty)\setminus\N$, or
        \item $W$ is a $\beta$-stable Lévy process with $\beta \in (0,2]$.
    \end{itemize}    
\end{assumption}

The above class of processes includes many different behaviours, including both Markovian and non-Markovian processes, which can be either of pure jump type or H\"older continuous with arbitrary exponent. Let us stress that Brownian motion corresponds to both $H=1/2$ and $\beta=2$.

Given $W$ satisfying Assumption~\ref{ass:noise}, we correspondingly define a regularity threshold $\alpha_W\geq 0$ by
\begin{equation}\label{eq:noise_regularity_threshold}
    \alpha_W:=
    \begin{cases}
        \max\{1-\frac{1}{2H},0\} \quad & \text{if $W$ is a fBm,}\\
        1-\frac{\beta}{2} & \text{if $W$ is a Lévy process.}
    \end{cases}
\end{equation}
It is known that, if $u\in C^\alpha([0,1]\times\T^2)$ with $\alpha>\alpha_W$, then for any noise intensity $\kappa>0$, the SDE~\eqref{eq:intro_SDE}
is strongly well-posed in a \emph{regularization by noise} fashion; see Section~\ref{sec:regbynoise_overview} for a general overview.
In this case, we denote by $X^\kappa_t(x,\omega)$ the stochastic flow associated to~\eqref{eq:intro_SDE}, defined on the canonical probability space $(\Omega,\cF,\P)$.
We study the vanishing noise limit $\kappa \to 0$ of the unique stochastic flow and prove that, in general, it does not converge to a unique deterministic flow associated with~\eqref{eq:intro_SDE} when $\kappa=0$, for velocity fields of H\"older regularity with $\alpha<1$.

\begin{theorem}
    \label{th:main_theorem}
    Let $W$ satisfy Assumption~\ref{ass:noise} and $\alpha_W$ be given by~\eqref{eq:noise_regularity_threshold}.
    Then for any $\alpha \in (\alpha_W,1)$ and $\eps \in (0,1)$, there exist a divergence-free drift $u \in C^\alpha ([0,1]\times \T^2;\R^2)$, a set $A_\eps \subset \T^2$ with $|A_\eps| \geq 1 - \eps $, a vanishing sequence $(\kappa_q )_{q}$, and a constant $c_\eps >0$ such that
    \begin{equation}\label{eq:main_theorem}
        \P\left( \liminf_{q\to\infty}  \dist  (X_1^{\kappa_{2q}} (x, \omega) , X_1^{\kappa_{2q+1}} (x, \omega))  \geq c_\eps> 0\ \text{ for Lebesgue a.e } x\in A_\eps\right)=1.
    \end{equation}
    Furthermore, for any fixed $x \in A_\eps$, one can find a subsequence $\{\kappa_{q'}\}_{q'}$ and deterministic points $y,y'$ in $\T^2$ such that
    \begin{equation}\label{eq:main_theorem_eq2}
        \P\left(\lim_{q'\to\infty} X_1^{\kappa_{2q'}} (x, \omega)=y\neq y'=\lim_{q'\to\infty} X_1^{\kappa_{2q'+1}} (x, \omega)\right)=1.
    \end{equation}
\end{theorem}

\begin{remark}
    Let us stress several features and consequences of Theorem~\ref{th:main_theorem} and the associated construction of $u$:
    \begin{itemize}
        \item[i)] For any $x\in A_\eps$, the ODE admits multiple trajectories starting from $x$.
        \item[ii)] By~\eqref{eq:main_theorem}, for $\P$-a.e. $\omega\in \Omega$, simultaneously for Lebesgue a.e. $x\in A_\eps$, the sequence $\{X^{\kappa_q}_1(x,\omega)\}$ is not Cauchy and does not converge to a unique limit. In particular, one cannot expect selection to happen in the sense of \emph{Regular Lagrangian Flows}.
        \item[iii)] In fact, even selection at the level of laws is not possible: for every $x\in A_\eps$, the sequence $Law(X^\kappa_t(x))$ is tight but does not admit a unique limit; see Corollary~\ref{cor:lack_selection_laws} for the precise statement.
        \item[iv)] In the last part of Theorem~\ref{th:main_theorem}, the points $y$, $y'$ are deterministic. Indeed, one can extract subsequences $(\kappa_{q'})_{q'}$ such that $X^{\kappa_{q'}}_t(x,\omega)$ converge $\P$-a.s. to distinct deterministic limits, excluding the necessity to see true superposition solutions as $\kappa\to 0^+$; see Proposition~\ref{prop:preliminary_main_thm} and Corollary~\ref{cor:vanishing_variance} for more precise statements.
    \end{itemize}
\end{remark}

\begin{remark}
    In order for the stochastic flow $X^\kappa(x,\omega)$ to exist, condition $\alpha>\alpha_W$ appears to be necessary.
    Very recently, counterexamples to pathwise uniqueness have been shown in~\cite{rowan2026}, for SDEs driven by fBm of parameter $H\in (0,1)$, in the full range $\alpha<1-\frac{1}{2H}$.
\end{remark}

In recent years, there has been an increasing interest in understanding the behaviour of solutions $X^\kappa$ in the vanishing noise limit $\kappa\to 0^+$, due to its connection to phenomenologies of \emph{passive scalar turbulence}. In the Brownian case $\beta=2$, whenever $u$ is divergence-free and regular enough (e.g. $u$ continuous and bounded), it is well-known that the solution $\rho^\kappa$ to the advection-diffusion PDE\footnote{The coefficient $\kappa^2/2$ in front of the Laplacian, instead of the usual $\kappa$, is due to our convention on the noise intensity $\kappa$ in the SDE~\eqref{eq:intro_SDE}.}
\begin{equation}\label{eq:intro_PDE}
    \partial_t \rho^\kappa + u\cdot\nabla \rho^\kappa = \frac{\kappa^2}{2} \Delta \rho^\kappa, \quad \rho\vert_{t=0}=\rho_{in}
\end{equation}
can be represented by means of the associated stochastic flow $X^\kappa$; specifically, one has
\begin{equation}\label{eq:formula-fractional-lap}
    \rho_{\kappa}(t,x) 
    = \mathbb{E}\big[((X_t^\kappa)_\sharp \rho_{in})(x)\big]
    = \E\big[\rho_{in}((X^\kappa_t)^{-1}(x))\big].
\end{equation}
A link between the properties of \emph{anomalous dissipation} of energy of $\rho^\kappa$, \emph{spontaneous stochasticity} of $X^\kappa$ and non-uniqueness results has been observed in \cite{DriEyi2017,Rowan2024}, based on Fluctuation-Dissipation relation and the superposition principle: there exist $\rho_{in}\in L^2(\T^d)$ and a sequence $\kappa_q\to 0^+$ such that
\begin{align*}
    \lim_{q\to\infty} \kappa_q^2 \int_0^1 \| \nabla \rho^{\kappa_q}_t\|_{L^2}^2 \dd t>0
\end{align*}
if and only if there exists a set $C\subset \T^d$ of positive Lebesgue measure and a sequence $\kappa_{q'}\to 0^+$ such that $Law(X^{\kappa_{q'}}(x))$ does not converge to a Dirac mass as $q'\to\infty$, for every $x\in C$.
In particular, anomalous dissipation implies both non-uniqueness of weak solutions to the backward inviscid transport PDE and non-uniqueness of backward trajectories for the ODE problem, for any $x\in C$.
We will comment more thoroughly on this point in Section~\ref{subsec:literature}; for the moment, let us point out that, as a consequence of representation formulas in the style of~\eqref{eq:formula-fractional-lap} (see Theorem~\ref{thm:SWP_stable}) and Theorem~\ref{th:main_theorem}, we obtain the following corollary about non-selection for PDEs. It generalizes the result from~\cite{CCS2023} to the case of vanishing fractional viscosity, for any exponent $\beta\in (0,2]$. However, a key distinction from that result is that our non-selection criterion yields two distinct conservative solutions in the vanishing-diffusivity limit, whereas the result in \cite{CCS2023} yields one conservative solution and one dissipative solution.

\begin{corollary}
\label{cor:vanishing-viscosity-PDE}
Let $\beta\in (0,2]$ fixed. Then for any $\alpha \in (\alpha_W,1)$, there exist a divergence-free drift $v \in C^\alpha ([0, 1]\times \T^2;\R^2)$, a vanishing sequence $(\kappa_q)_q$ and a non-negative, smooth initial condition $\rho_{in} \in C^\infty(\T^2;[0,+\infty))$ such that the sequence $(\rho^{\kappa_q})_q$ of solutions to the fractional advection-diffusion PDE
\begin{equation}
\label{eq:PDE_fractional_laplacian}
\partial_t \rho^{\kappa_q} + v \cdot\nabla \rho^{\kappa_q}  = -\kappa_q^\beta |\nabla|^\beta \rho^{\kappa_q}, \quad
\rho^{\kappa_q}\vert_{t=0} = \rho_{in}
\end{equation}
does not admit a unique limit as $q\to\infty$; more precisely, there exist two distinct solutions $\rho, \bar\rho$ to the transport PDE
\begin{equation}\label{eq:transport_PDE}
    \partial_t \rho + v\cdot\nabla \rho = 0, \quad
\rho\vert_{t=0} = \rho_{in}
\end{equation}
such that
$$\rho^{\kappa_{2q}}   \overset{\ast}{\rightharpoonup} \rho \neq \bar \rho  \overset{\ast}{\leftharpoonup} \rho^{\kappa_{2q+1}}$$
where the convergence is weak-$\ast$ in $L^\infty_{t,x}$. Furthermore, along the sequence $(\kappa_q)_q$ there is no anomalous dissipation of energy, namely
$$ \lim_{q\to \infty} \kappa_q^\beta \int_0^1 \| |\nabla |^{\beta/2} \rho^{\kappa_q}_t \|_{L^2}^2 \dd s =0 \,.$$
\end{corollary}

\subsection{Main ideas of proof}

The construction of the velocity field is based on a careful refinement of the one originally presented in~\cite{CCS2023}.
In particular, $u$ is first constructed on the time interval $[0,1/2]$, in such a way that it produces smaller scale structures starting from larger ones via a mixing mechanism (one can capture this by looking at the evolution of chessboards, see Figure~\ref{fig:chessboards}), with  almost perfect mixing occurring at $t=1/2$. Loosely speaking, up to the presence of appropriately rescaled mollifications, the velocity field takes the form
\[
u(t,x,y) \simeq \sum_{q=1}^{\infty} v_{q}(t,x,y)\,
\mathbbm{1}_{[t_q,t_{q+1}]}(t),
\qquad t\in (0,1/2),
\]
where \(t_q \uparrow 1/2\) and $v_q(t,\cdot)$ is either a horizontal or vertical shear flow at any fixed time $t$.
Moreover, \(v_q\) is essentially localized at Fourier frequencies with \(k\sim a_{q+1}^{-1}\), where $a_q$ is a super exponentially decreasing sequence; namely we have 
\[
\Pi_{\simeq a_{q+1}^{-1}} v_q \simeq v_q,
\]
where $\Pi_{\simeq \lambda}$ denotes the Fourier projection onto frequencies of order $|k|\sim\lambda$.
By carefully tuning the time, amplitude and frequency parameters, one can obtain in this way a drift $u\in C^\infty([0,1/2)\times\T^2)$, with a singularity at $t=1/2$, such that $u\in C^\infty([0,1/2];C^\alpha(\T^2))$ (cf. Lemma~\ref{lemma:regularity-holder}).
On the interval \((1/2,1)\), one would like to reflect the construction in time in order to produce an unmixing mechanism, so that small scales get transferred to large scales.
However, in order to exploit and amplify the instabilities already present due to mixing, with the goal in mind of showing non-uniqueness of integral curves obtained in the zero noise limit, it turns out convenient to add a small perturbation in this reflection, by adding a ``swap'' velocity field.
As a consequence, not only $u$ admits non-unique trajectories for a large set of initial conditions, but one can find smooth approximations $u_q$ of $u$ which exhibit different behaviours depending on whether $q$ is even or odd (cf. Section~\ref{subsec:properties_lagrangian_flow}). The smooth approximations $u_q$ are obtained by localizing $u$ away from the singular time $t=1/2$, which by construction is roughly equivalent to the application of a Fourier cutoff:
\[
u_q = u (t,x) \cdot \one_{[t_q, 1-t_q]^c} (t) \simeq \Pi_{\leq a_q^{-1}} u.
\]
The aforementioned parity-dependent behaviour is due to the swap velocity field, which induces a highly unstable mechanism for reconstructing the large-scale structures of the unique Lagrangian flows $X_t^q$ associated to $u^q$, after the singular time $t=1/2$.

The next main step is to identify a super exponentially decaying sequence $(\kappa_q)_q$ of noise intensities, such that the stochastic flow \(X_t^{\kappa_q}\) associated with \(u\) gets quantitatively close to the deterministic one \(X_t^{q}\) associated to $u_q$ as $q\to\infty$; see Proposition~\ref{prop:preliminary_main_thm} for a more precise statement.
The rough idea is that the cancellations produced by the presence of noise $W$ (which are responsible for the well-posedness of the SDE) are comparable to a coarse-graining effect on the drift $u$, which allows to replace it by $u_q$, as long as $\kappa \simeq \kappa_q$ is appropriately chosen.
Once these estimates are obtained, the non-convergence of $X_t^{\kappa_q}$ to a unique flow, as stated in Theorem \ref{th:main_theorem}, follows from the aforementioned properties of $X^q$.

To obtain quantitative estimates, one needs to distinguish three regimes, depending on whether $|1/2-t| \gg 1/2-t_q$, $|1/2-t|\simeq t_q$ or $|1/2-t| \ll 1/2-t_q$.

In the first regime, the velocity field $u$ is regular relatively to noise intensities $\kappa\ll a_q$ and one can perform pathwise estimates (see Proposition~\ref{prop:up_to_tq}). Still, the proof is quite delicate: a direct Gr\"onwall argument would produce an exponential factor $\int_0^{t_q} \|\nabla u(s)\|_{L^\infty} \dd s$, which is too large; it is therefore necessary to obtain a substantial improvement over this naive estimate. This is achieved by exploiting the geometry of our construction:  $u$ consists of alternating shear flows, it is locally constant on a set of large Lebesgue measure and its gradient vanishes on most of the domain.

In the second regime, we need to use the fact that $u$ is concentrated on high Fourier modes, combined with the probabilistic cancellations induced by the noise $W$.
To achieve this, we exploit the Stochastic Sewing Lemma (SSL) introduced by L\^e~\cite{le2020sewing}, see Section~\ref{sec:stochastic_estimates}; this step can be rigorously carried out whenever $\kappa \gg a_{q+1} $.
The alternating shear structure of $u$ is again fundamental, as it allows to reduce the relevant estimates to manipulations of one-dimensional functions. It is interesting to note how the SSL, which usually plays a pivotal role to obtain positive regularization by noise results, here turns out to be useful to prove negative non-selection results instead.

The last regime is the simplest one, as one can perform direct estimates using the smallness of $\| u(t)\|_{L^\infty}$ relative to $\kappa_q$ in this range of times.

The super-exponential choice of the sequences of noise intensities $\kappa_q$ and amplitudes $a_q$ is crucial in order to satisfy the conditions on $\kappa$ arising from these regimes and complete the proof.

We conclude this section by discussing the main differences and innovations of this work, compared to previous constructions based on alternating shear flows like e.g.~\cite{DEIJ22,CCS2023,EL23} and partially~\cite{AV23}:
\begin{itemize}
    \item In light of the representation formula~\eqref{eq:formula-fractional-lap} and tightness arguments, any non-selection result for the vanishing viscosity limit of advection-diffusion PDEs necessarily implies lack of selection in law at the level of (backward) stochastic Lagrangian flows with Brownian noise, despite not being explicitly stated; this is the case for~\cite{CCS2023,AV23}. However this consequence is mostly qualitative, providing relatively little information about the behaviour of $X^\kappa$ in the limit. In comparison, Theorem~\ref{th:main_theorem} is rather explicit and allows to deduce several interesting consequences.
    \item We adopt an \emph{entirely} Lagrangian approach, not relying on PDE methods, which makes our construction more flexible. In particular, it can be applied to a large class of stochastic perturbations $W$, not necessarily Markovian; a key step is to replace estimates based on the It\^o--Tanaka trick with stochastic sewing ones. Moreover, differently from~\cite{CCS2023}, the arguments (especially Proposition~\ref{prop:up_to_tq}) are carefully carried so to avoid any continuity and/or exponential integrability assumption on $W$, which allows one to include $\beta$-stable noises of any order $\beta\in (0,2]$.
    \item Finally, the same non-selection statement can be lifted to three-dimensional autonomous velocity fields, see~Theorem \ref{thm:autonomous_case}. Despite the apparent simplicity of the idea (treat time as the third dimension), the argument is rather non-trivial. It requires to exploit the higher time regularity of $u$, obtained by choosing $t_q$ to be a ``slow'', polynomially convergent sequence (following~\cite{EL23}), with improved a priori estimates for the SDE and shifted SSL arguments (inspired by~\cite{Gerencser2023,GalGer2025}).
\end{itemize}

\subsection{Relations with existing literature}\label{subsec:literature}

As mentioned, in dimension $d=1$, the vanishing noise selection problem was investigated by Bafico and Baldi in~\cite{BafBal1982} in the case of Brownian SDEs; typical examples of velocity fields included in their analysis are the ones behaving around the origin like $u(x)={\rm sgn}(x) |x|^\gamma$ with $\gamma\in (0,1)$. The result has been recovered using a more dynamical approach in~\cite{DelFla2014} and local times in~\cite{Trevisan2013}. Singular large deviations have been subsequently established in~\cite{GHR2001}; similar results have been recently obtained in the case of fBm drivers in~\cite{MadGas2024}.
Extensions to higher dimensions are rather incomplete, see~\cite{DelMau2019,PP2018,KulPil2020}.

A non-Lipschitz scenario in which the vanishing noise limit is well-understood is given by the DiPerna--Lions theory of Regular Lagrangian Flows (RLFs), which applies in any dimension\footnote{The interesting case is $d\geq 2$, otherwise condition $\nabla\cdot u\in L^1([0,T];L^\infty)$ reduces to the standard Lipschitz one.} under (for instance) the regularity assumptions $u\in L^1([0,T];W^{1,1})$, $\nabla\cdot u\in L^1([0,T];L^\infty)$.
It was already pointed out in~\cite[Sec. IV.3]{DiPLio1989} that in this case both the unique Regular Lagrangian Flow to the ODE and renormalized solutions to the transport PDE are recovered in the vanishing noise limit. Large deviation results have been established in~\cite{Zhang2013}, stability estimates in~\cite{LL2019} and quantitative rates in~\cite{BoCiCr2022}. All the aforementioned results apply in the Brownian case, but it was pointed out in~\cite{Galeati2025} how pathwise arguments imply stability results and therefore the characterization of the vanishing noise limit for essentially \emph{any} choice of additive noise $W$. Qualitative convergence results extend to the $BV$ setting treated by Ambrosio~\cite{Ambrosio2004}; more generally, as soon as the velocity field $u$ satisfies the renormalization property, the associated RLF is recovered in the vanishing noise limit and both features of anomalous dissipation of energy and spontaneous stochasticity are ruled out, as illustrated in~\cite[Sec. 3]{BBDM2026}.

For divergence-free vector fields 
$u\in L^1_{\mathrm{loc}}((0,T];BV)$, allowing for a singularity at time $t=0$, a characterization of the vanishing noise limit is still possible, both at the level of the transport PDE and in terms of the limits $\tilde\nu_x$ of $Law(X^\kappa(x))$ as probability measures on the path space; see the series of works~\cite{Pitcho2024,MPS2025,Pitcho2026,Pitcho2026b}. It is worth noting that, for Lebesgue a.e. $x$, the measure $\tilde\nu_x$ obtained is a true superposition, cf.~\cite[Thm. 1.5-(ii)]{Pitcho2026}. The class of vector fields $u$ considered in these works includes as a special case De Pauw's example~\cite{Depauw2003}.

Concerning results for PDEs, let us mention: \cite{AttFla2009} for the study of the stochastic transport equation in the Bafico-Baldi regime, showing that the limit $u$ consists of a true superposition solution to the transport PDE, whose expectation $\E[u]$ corresponds to the unique vanishing viscosity limit; \cite{DLN2001}, focusing on the continuation of a mean curvature flow in the case of fattening; \cite{DFV2014}, for the selection of a non-Markovian continuation of two charged particles in the Vlasov--Poisson system after collapse; \cite{GRV2024}, for numerical study of continuation after collapse for three-point vortex dynamics in $2$D Euler.

Small noise perturbations are not the only possible option in order to devise selection criteria of solutions for ODEs; see for instance the work of Drivas and Mailybaev~\cite{DriMai2021}, based on smoothing and the use of dynamical tools (attractors and renormalization). On the other hand, the counterexamples from~\cite{CCS2020,DeLGir2022,CCS2023} show that smoothing the velocity field in general does not provide a selection criterion at either the Lagrangian or Eulerian levels.

As discussed around~\eqref{eq:formula-fractional-lap}, when $W$ is a Brownian noise, understanding the behaviour of the stochastic flow $X^\kappa$ as $\kappa\to 0^+$ is closely tied to the properties of solutions to the advection-diffusion equation~\eqref{eq:intro_PDE} and therefore also vanishing viscosity solutions to the inviscid transport PDE (possibly obtained as weak limits after extraction of a subsequence).
Starting with~\cite{DEIJ22}, several authors have constructed H\"older continuous velocity fields, whose solutions can exhibit pathological behaviours in the vanishing viscosity limit.
In~\cite{CCS2023}, Colombo, Crippa, and the third author of the present manuscript first presented a drift $u$ for which anomalous dissipation of energy and lack of unique selection both occur as $\kappa\to 0^+$.
In a similar setting, Huysmans and Titi construct in~\cite{HuyTit2025} vanishing viscosity solutions whose energy profile
$
t\mapsto \|\rho(t)\|_{L^2}
$
is not necessarily monotone, namely such that $\| \rho (2) \|_{L^2}>\| \rho (1) \|_{L^2}$.
Such phenomena can occur because vanishing viscosity solutions need not converge strongly in $L^2$; in general, only weak-$*$ convergence is available.
Lack of selection and anomalous dissipation in the vanishing viscosity limit are caused by the strong instability of small scales, causing propagation of singularities and forward cascade of energy, to the point where persistence of energy dissipation occurs independently of the diffusivity. 
Without attempting to mention all recent contributions on this topic, we refer to~\cite{CCS2023,JS24,AV23,BSW23,EL23,DEIJ22,CR25}.

\subsection{Structure of the paper}
Section~\ref{sec:regbynoise_overview} recalls known results about the stochastic processes we consider and the associated regularization by noise results.
Section~\ref{sec:stochastic_estimates} presents some stochastic estimates obtained by stochastic sewing techniques, which will be crucial in the proof of our main result.
The construction of the velocity field $u$ is given in Section~\ref{sec:construction}, together with the properties of the flows $X^q$ associated to the smooth approximations $u_q$.
Section~\ref{sec:stability} is devoted to the proof of a pathwise stability estimate between $X^{\kappa_q}$ and $X^q$, on well-chosen time intervals $[0,t_q]$ on which $u$ is smooth relative to the size of the noisy perturbation $\kappa_q$.
We combine all these ingredients in Section~\ref{sec:proof_main_thm}, which presents the proof of Theorem~\ref{th:main_theorem}.
Section~\ref{sec:consequences} discusses some consequences of our construction, specifically: a density result for drifts with the non-selection property (Corollary~\ref{cor:density-non-selection}); the proof of Corollary~\ref{cor:vanishing-viscosity-PDE}; the construction of an autonomous drift $u$ with the same property in dimension $3$ or higher (Theorem~\ref{thm:autonomous_case}).
Appendix~\ref{app:useful} recalls some useful results about smoothing estimates for Markov kernels associated to Lévy processes.

\subsection{Frequently adopted notations} 
We conclude this introduction with the main notations adopted in this manuscript.
We denote by $\T^d \sfrac{\cong [0,1]^d }{\sim}  $  the $d$-dimensional torus. We use the convention $\inf \emptyset = - \infty $. For $k \in \N$, we denote $k\N$ for the set of natural numbers that are multiples of $k$. Moreover, in this paper we will use the following notations.

\begin{itemize}
    \item We work on the time interval $[0,1]$. Any other interval $[S,T]$ appearing in the following is a subset of it.
    \item When deriving estimates, we write $a\lesssim b$ to indicate that there exists a positive constant $C$ such that $a \leq C b$.
    \item We write $D f$ for the differential (in space, possibly in a distributional sense) of $f$.
    \item We consider inhomogeneous Besov spaces $B^\theta_{p,q}(\R^d)$ as defined by Littlewood-Paley blocks $\Delta_j$, see~\cite{BCD2011}; these definitions transfer to the torus $\T^d$ by Poisson summation formula. Whenever clear, we will just write $C^{\theta}_x$ to denote $B^{\theta}_{\infty,\infty}(\T^d;\R^m)$ for $\theta \in \R\setminus \N$, with norm $\| \cdot \|_{C^\theta_x}$. Note that, for $\theta\in (0,+\infty)\setminus \N$, this is consistent with the use of standard H\"older spaces.
    \item For $k\in\N$, we use instead $C^k_x=C^k(\T^d;\R^m)$ to denote the Banach space of continuous functions with continuous derivatives up to order $k$, with norm $\| \cdot\|_{C^k_x}$.
    \item Given an interval $I\subset [0,1]$, $p\in [1,\infty]$ and a Banach space $E$, we denote by $L^p(I;E)$ the associated Bochner-Lebesgue space of functions $f:I\to E$. Whenever $I$ is clear, we may adopt the shortcut notations $L^p_t E$, $\| \cdot\|_{L^p_t E}$. Similar considerations apply for the spaces $C^k(I;E)$, consisting of functions $f$ whose times derivatives $\partial_t^j f$ for $j\leq k$ are continuous maps from $I$ to $E$; whenever convenient we denote its norm by $\| \cdot\|_{C^k_t E}$.
    \item We denote $f\star g$ for convolutions (both on $\R^d$ and $\T^d$).
    \item Given a set $A\subset \T^d$, $|A|$ denotes its Lebesgue measure.
    \item We use $\dist(x,y)$ to denote the distance on $\T^d$.
    \item If $a, b \in \R$, we set $a \wedge b = \min(a, b)$ and $a \vee b = \max(a, b)$.
\end{itemize}

\section{Stochastic processes and regularization by noise results}\label{sec:regbynoise_overview}

In what follows, we fix a reference (usually canonical) probability space $(\Omega, \cF,\P)$, on which a càdlàg process $(W_t)_{t\in [0,1]}$ is defined.\footnote{In this paper we always work on the interval $[0,1]$, although the content of this section is true for any bounded interval $[0,T]$.}
We denote by $(\cF_t)_{t\geq 0}$ generic filtrations on $(\Omega, \cF,\P)$, which are always assumed to satisfy the usual assumptions. $(\cF^W_t)_{t\geq 0}$ stands for the (usual augmentation of the) natural filtration generated by $W$.

Given an integrable random variable $Z$ and a filtration $(\cF_t)_{t\geq 0}$, we will adopt $\E_s Z$ as a short-hand notation for the conditional expectation $\E[Z | \mathcal{F}_s]$. Whenever the filtration is not specified, it means we are considering $(\cF^W_t)_{t\geq 0}$.
We will focus on the two main cases for $W$: fractional Brownian motion and stable L\'evy processes.

\subsection{Fractional Brownian motion}\label{subsec:fbm}

We say that $(W^H_t)_{t\in [0,T]}$ is a fractional Brownian motion (fBm) with Hurst parameter $H\in (0,1)$ if it is a continuous centered Gaussian process with covariance
\begin{equation*}
    \E[W^H_t W^H_s]=\frac{1}{2}(t^{2H}+s^{2H}-|t-s|^{2H}).
\end{equation*}
We can extend the definition recursively to any $H\in(0,+\infty)\setminus\N$ by $W^{H+1}_t=\int_0^t W^H_r \dd r$ and to $d$-dimensional processes by requiring each coordinate to be an i.i.d. fBm.

For $H=1/2$, fBm coincides with standard Brownian motion. However, for $H\neq 1/2$, it is not a semimartingale nor a Markov process; we refer to~\cite{GalGer2025} and the references therein for a deeper discussion.
In particular, $W^H$ does not have independent increments and there is no natural Markov transition kernel underlying it. It does however possess a key property of \emph{local nondeterminism} (LND), which we will employ later; to explain it, we need a few notations.

We denote by $\mathcal{P}_t=e^{-t \Delta}$ the standard heat kernel on $\R^d$; equivalently, one has $(\mathcal{P}_t f)(x)=(p_t\star f)(x)$, where $p_t=(4\pi t)^{-d/2} \exp(-\sfrac{|x|^2}{4t})$, and $\widehat{\mathcal{P}_t f}(\xi)=e^{-t|\xi|^2} \widehat f(\xi)$.
The LND of fBm can then be stated as follows: there exists a constant $c_H>0$ such that, for any $s\leq t$, any bounded measurable function $f:\R^d\to \R$ and any other $\cF^W_s$-measurable random variable $Z$, it holds that
\begin{equation}\label{eq:LND_fBm}
    \E[f(Z +W^H_t)|\cF^W_s]= (\mathcal{P}_{c_H |t-s|^{2H}} f)(Z + \E[ W^H_t|\cF^W_s]);
\end{equation}
see for instance~\cite[eq. (1.25)]{GalGer2025}.

Following~\cite[Def. 1]{NuaOuk2002}, given a filtration $(\cF_t)_{t\geq 0}$ (potentially larger than $(\cF^W_t)_{t\geq 0}$), one can introduce the concept of $(\cF_t)_{t}$-fBm, in which case relation~\eqref{eq:LND_fBm} is still valid with $\cF^W$ replaced by $\cF$; to avoid too many technicalities, we refrain from giving the full definition of $(\cF_t)_{t}$-fBm here.

The strong regularization by noise properties of fBm on ODEs are summarized in the following statement from~\cite{GalGer2025}; see also~\cite{CatGub2016,le2020sewing,Gerencser2023} for relevant precursors. 
In the following theorem, a stochastic flow of diffeomorphisms is understood in the following sense: there exists a set $\Gamma \subset \Omega$ of full probability such that, for every $\omega \in \Gamma$, the following hold:
\begin{itemize}
\item the map $(s,t,x) \mapsto X^\kappa_{s,t}(x,\omega)$ 
is jointly continuous in $(t,x)$ and differentiable in $x$, with $(s,t,x) \mapsto D X^\kappa_{s,t}(x,\omega)$
also jointly continuous in $(t,x)$;
\item $X^\kappa_{s,s}(x,\omega)=x$, and for every $s<u<t$, the semigroup property holds:
\[
X^\kappa_{u,t}(\cdot,\omega)\circ X^\kappa_{s,u}(\cdot,\omega)=X^\kappa_{s,t}(\cdot,\omega)\,;
\]
\item for every $s<t$, the map $x \mapsto X^\kappa_{s,t}(x,\omega)$ is a diffeomorphism of $\T^d$.
\end{itemize}

\begin{theorem}\label{thm:SWP_fBm}
    Let $W^H$ be an fBm of parameter $H\in (0,+\infty)\setminus\N$, $\kappa>0$ and let $u\in L^\infty([0,1];C^\alpha(\T^d))$ with\footnote{Condition $\alpha\geq 0$ in~\eqref{eq:condition_fBm_SWP} is not even needed in~\cite{GalGer2025}, which is set on $\R^d$; we only enforce it in order to avoid technicalities concerning the interpretation one must give to the integral in~\eqref{eq:SDE_fBm:SWP} and the identification of elements of $C^{-\theta}(\T^d)$ with suitable elements of $C^{-\theta}(\R^d)$.}
    \begin{equation}\label{eq:condition_fBm_SWP}
        \alpha\geq 0, \quad \alpha>1-\frac{1}{2H}.
    \end{equation}
    Then, for any $x\in\T^d$, there exists a strong, pathwise unique solution to the SDE
    \begin{equation}\label{eq:SDE_fBm:SWP}
        X^\kappa_t = x + \int_0^t u(r,X^\kappa_r) \dd r + \kappa W^H_t.
    \end{equation}
    Moreover, the SDE admits a stochastic flow of diffeomorphisms $\{X^\kappa_{s,t}(x,\omega)\}$ on $\T^d$.
\end{theorem}

Theorem~\ref{thm:SWP_fBm} is a consequence of~\cite{GalGer2025}, up to allowing an intensity $\kappa>0$ in front of the noise (which can be easily reabsorbed by a time rescaling argument, using self-similarity of fBm) and by regarding the velocity field $u$ and the resulting flow $X^\kappa$ as periodic functions on $\R^d$ (see also the proof of Theorem~\ref{thm:SWP_stable} below for more details).

\subsection{Stable Lévy processes}\label{subsec:stable_levy}

Given $\beta\in (0,2)$, the isotropic $\beta$-stable Lévy process $W^\beta$ on $\R^d$ is the Lévy process with triplet $(0,0,\nu^\beta)$ given by $\nu^\beta(\dd x) = c_{d,\beta} |x|^{-d-\beta} \dd x$; it is a Markov process with independent increments and generator given by the fractional Laplacian $\mathcal{L}_\beta=-|\nabla|^\beta=-(-\Delta)^{\beta/2}$, which can be equivalently defined in either real or Fourier variables as 
\begin{align*}
    (|\nabla|^\beta f)(x)= c_{d,\beta}\, {\rm p.v.} \int_{\R^d} \frac{f(x)-f(y)}{|x-y|^{d+\beta}} \dd y, \qquad \widehat{|\nabla|^\beta f}(\xi)=|\xi|^\beta \widehat{f}(\xi) 
\end{align*}
whenever $f$ is regular enough.
$W^\beta$ is a pure jump process with $W^\beta_0=0$; in particular, with probability $1$ its paths $t\mapsto W^\beta_t$ are càdlàg, locally constant and discontinuous on every interval. By~\cite[Thm. 25.3, Ex. 25.10 and Thm. 25.18]{Sato2013}, it satisfies
\begin{equation}\label{eq:moments_stable}
    \E\left[\sup_{t\in [0,1]} |W^\beta_t|^\eta\right]<\infty \quad\forall\, \eta\in (0,\beta), \qquad 
    \E[|W^\beta_t|^\beta]=+\infty \quad\forall\, t>0.
\end{equation}
We refer to the monographs~\cite{Applebaum2009,Sato2013} for more details on Lévy processes.
By standard multiplier theorems (cf.~\cite[Lem. 2.2]{BCD2011}), $|\nabla|^\beta$ is a bounded operator from $B^\theta_{p,q}$ to $B^{\theta-\beta}_{p,q}$ for any $\theta,p,q$.
We will denote by $\{\mathcal{P}^\beta_t\}_{t\geq 0}$ the Markov transition semigroup associated to $W^\beta$, given by $\mathcal{P}^\beta_t=e^{-t|\nabla|^\beta}$, namely
\begin{align*}
    \widehat{\mathcal{P}^\beta_t f}(\xi)=e^{-|\xi|^\beta} \widehat f (\xi).
\end{align*}
We refer to Appendix~\ref{app:useful} for the smoothing properties of the semigroup $\mathcal{P}_t^\beta$.

Given a filtration $(\cF_t)_t$, we say that $W^\beta$ is a $(\cF_t)_t$-$\beta$-stable process if $W^\beta$ is $(\cF_t)_t$-adapted and $W^\beta_t-W^\beta_s$ is independent of $\cF_s$ for any $s\leq t$. Under this assumption, in analogy with~\eqref{eq:LND_fBm}, for any $s\leq t$, any bounded measurable function $f:\R^d\to \R$ and any other $\cF_s$-measurable random variable $Z$, we have the (slightly simpler) formula
\begin{equation}\label{eq:LND_stable}
    \E[f(Z + W^\beta_t)|\cF_s]= (\mathcal{P}^\beta_{t-s} f)(Z + W^\beta_s).
\end{equation}

All the above definitions are given in full space, but they are consistent with the corresponding ones on $\T^d$ defined via discrete Fourier transform, e.g. $\widehat{|\nabla|_{\T^d}^\beta f}(k)=|k|^\beta \widehat f(k)$ for $k\in(2\pi\Z)^d$, as can be readily checked using Poisson's summation formula. For instance, given $f\in C^\gamma(\T^d)$ with $\gamma>0$, we can regard it as a $2\pi$-periodic element of $C^\gamma(\R^d)\hookrightarrow B^\gamma_{\infty,\infty}$; for $\gamma>\beta$, it follows that $|\nabla|^\beta f\in B^{\gamma-\beta}_{\infty,\infty}\hookrightarrow C^0(\R^d)$, so $|\nabla|^\beta f$ is classically pointwise defined and in fact still $2\pi$-periodic, thus can be identified with $|\nabla|_{\T^d}^\beta f\in C(\T^d)$.
For this reason, whenever it doesn't cause confusion, henceforth for simplicity we will not distinguish between $|\nabla|^\beta$ and $|\nabla|_{\T^d}^\beta$.

Let us denote by $\mathcal{M}(\T^d)$ the set of signed measures on $\T^d$. We write $\mu\in C_w([0,1];\mathcal{M}(\T^d))$ for the set of maps $t\mapsto \mu_t$ which are weakly continuous in the sense of measures; note that, by the  Banach-Steinhaus uniform boundedness principle, any $\mu\in C_w([0,1];\mathcal{M}(\T^d))$ satisfies $\sup_{t\in [0,1]} \| \mu_t\|_{TV}<\infty$, where $\| \cdot\|_{TV}$ stands for the total variation norm on $\mathcal{M}(\T^d)$.
Given $\nu\in\mathcal{M}(\T^d)$ and $f\in C(\T^d; \T^d)$, we denote by $f_\sharp \nu\in \mathcal{M}(\T^d)$ the pushforward of $\nu$ under $f$. 

The notion of stochastic flow adopted in Theorem ~\ref{thm:SWP_stable} is slightly more technical compared to Theorem~\ref{thm:SWP_fBm} due to potential jump discontinuities of $W^\beta$ and uncountable quantifiers.
    In particular, by~\cite[Thm. 5.1]{Priola2018}, one can find $\Gamma\subset \Omega$ of full probability such that, for any $\omega\in\Gamma$, we have that:
   \begin{itemize}
    \item for every $\omega \in \Gamma$, the map
    \[
    (s,t,x)\mapsto X^\kappa_{s,t}(x,\omega)
    \]
    is càdlàg in $(s,t)$ and continuous in $x$;

    \item $X^\kappa_{s,s}(x,\omega)=x$, and for every $s<u<t$, the semigroup property holds:
    \[
    X^\kappa_{u,t}(\cdot,\omega)\circ X^\kappa_{s,u}(\cdot,\omega)=X^\kappa_{s,t}(\cdot,\omega);
    \]

    \item  for every fixed $s$, by the results of~\cite{CSZ2018}, there exists a set $\Gamma_s$ of full probability such that, for every $\omega\in\Gamma_s$, the map $x\mapsto X^\kappa_{s,t}(x,\omega)$ is a diffeomorphism of $\T^d$ for every $t\ge s$;

    \item for the same $\omega\in \Gamma_s$, the map
    \[
    (t,x)\mapsto D X^\kappa_{s,t}(x,\omega)
    \]
    is jointly continuous.
\end{itemize}

Similarly to Theorem~\ref{thm:SWP_fBm}, we can summarize the regularizing features of $\beta$-stable paths, and in this case their connection to Fokker-Planck type PDEs, as follows.

\begin{theorem}\label{thm:SWP_stable}
    Let $W^\beta$ be an isotropic $\beta$-stable process with $\beta\in (0,2)$, $\kappa>0$ and let $u\in L^\infty([0,1];C^\alpha(\T^d))$ with\footnote{Compared to the majority of the literature on the topic, the notations for $\alpha$ and $\beta$ in~\eqref{eq:condition_stable_SWP} are reversed.}
    \begin{equation}\label{eq:condition_stable_SWP}
        \alpha>1-\frac{\beta}{2}.
    \end{equation}
    Then, for any $x\in\T^d$, there exists a strong, pathwise unique solution to the SDE
    \begin{equation}\label{eq:stable_SDE}
        X^\kappa_t = x + \int_0^t u(r,X^\kappa_r) \dd r + \kappa W^\beta_t
    \end{equation}
    and the SDE admits a stochastic flow of diffeomorphisms $\{X^\kappa_{s,t}(x,\omega)\}$.
    Furthermore, for any $\mu_0\in\mathcal{M}(\T^d)$, the map
    \begin{equation}\label{eq:FP_representation_formula1}
        \mu^\kappa_t:=\E\big[(X^\kappa_{0,t})_\sharp \mu_0\big] \quad\forall\, t\in [0,1]
    \end{equation}
    is the unique solution in $C_w([0,1];\mathcal{M}(\T^d))$ (in the weak sense) to the PDE
    \begin{equation}\label{eq:FP_PDE_stable}
        \partial_t \mu^\kappa + \nabla\cdot (u\, \mu^\kappa) = - \kappa^\beta |\nabla|^\beta \mu^\kappa,\quad\mu^\kappa\vert_{t=0}=\mu_0.
    \end{equation}
\end{theorem}

\begin{proof}
    Up to rescaling the time interval and using the self-similarity properties of $W^\beta$, we may restrict ourselves to $\kappa=1$; we will just write $X$ in place of $X^1$.
    
    We may identify $u$ with a $2\pi$-periodic element of $L^\infty([0,1];C^\alpha(\R^d))$; then strong existence, pathwise uniqueness and existence of a stochastic flow of diffeomorphisms on $\R^d$ under condition~\eqref{eq:condition_stable_SWP} follow from the works~\cite{Priola2012,CSZ2018,Priola2018}.
    From pathwise uniqueness and periodicity of $u$, for any fixed $(s,x)$ and any $k\in\Z^d$, one deduces that $\P$-a.s. $X_{s,t}(x+2\pi k,\omega)=X_{s,t}(x,\omega)$ for all $t\in [s,T]$.
    Using the $\P$-a.s. continuity properties of $(s,t,\omega)\mapsto X_{s,t}(x,\omega)$  and countable quantifiers, one can then find a common set $\tilde\Omega\subset \Omega$ of full probability on which $X_{s,t}(x+2\pi k,\omega)=X_{s,t}(x,\omega)$ for all $s,t,x,k$; therefore $X$ can be regarded as a stochastic flow on $\T^d$.
    
    Set $X^x_t(\omega):=X_{0,t}(x,\omega)$. By the It\^o formula for Lévy processes (cf.~\cite{Applebaum2009}), for any $\varphi\in C^2(\R^d)$ we have
    \begin{equation}\label{eq:ito_stable}
        \E[\varphi(X^x_t)] = \varphi(x) + \int_0^t \E[(u_r\cdot\nabla\varphi)(X^x_r)] \dd r + \int_0^t \E[(|\nabla|^\beta\varphi)(X^x_r)] \dd r \quad\forall\, t\geq 0 
    \end{equation}
    which shows that $\mu_t:=Law(X^x_t)=\E[\delta_{X^x_t}]=\E[(X_{0,t})_\sharp \delta_x]$ is a weak solution to~\eqref{eq:FP_PDE_stable} on $\R^d$ with $\mu_0=\delta_x$. By testing against $(2\pi)$-periodic functions $\varphi$, in view of the aforementioned identifications between $|\nabla|^\beta$ and $|\nabla|_{\T^d}^\beta$, it is also a weak solution on $\T^d$. Finally, it follows from~\eqref{eq:ito_stable} and standard density arguments that $\mu\in C_w([0,1];\mathcal{M}(\T^d))$.
    The same claim for general $\mu_0\in \mathcal{M}(\T^d)$ then follows from the linearity of the map $\mu_0\mapsto (X_{0,t})_\sharp \mu_0$ and disintegration arguments.

    It remains to establish uniqueness of weak solutions to~\eqref{eq:FP_PDE_stable} in the class $C_w([0,1];\mathcal{M}(\T^d))$.
    The argument is classical (in fact valid under the weaker assumption $\alpha>1-\beta$), but the statement is hard to find explicitly in the literature, so we shortly sketch it.
    By linearity, it suffices to show that $\mu\equiv 0$ if $\mu_0\equiv 0$.
    Regarding $u$ as a $2\pi$-periodic element of $L^\infty([0,1];C^\alpha(\R^d))$, by~\cite[Thm. 2.3 \& Cor. 2.10]{CSZ2018} (applied with $\lambda=0$, $\delta=\bar\alpha=1$), for any $g\in L^\infty([0,1];C^\alpha(\R^d))$ there exists a unique classical solution $f$ to the dual backward equation
    \begin{equation*}
        \partial_t f + u\cdot\nabla f = |\nabla|^\beta f + g, \quad f\vert_{t=T}=0
    \end{equation*}
    which belongs to $L^\infty([0,1];C^\gamma(\R^d))$ for any $\gamma<(\alpha+\beta)\wedge 2$ (for $\beta\in (0,1)$, see also the sharper Schauder estimates obtained in~\cite{CMP2020}). In particular, when $g$ is $2\pi$-periodic, then by uniqueness so is $f$, which can be regarded as a function on $\T^d$.
    Thanks to the regularity of $f$, the distributions $\partial_t f$, $u\cdot\nabla f$ and $|\nabla|^\beta$ all belong to $L^\infty([0,1];C^0(\T^d))$ and so (possibly after an approximation argument) we can test $\mu$ against $f$ to find
    \begin{align*}
        0 = \langle \mu_T, f_T\rangle-\langle\mu_0,f_0\rangle=\int_0^T \langle \mu_r, g_r\rangle \dd r.
    \end{align*}
    As the argument works for any $g\in L^\infty([0,1];C^\beta(\T^d))$, by the fundamental lemma of the calculus of variations and the weak continuity of $t\mapsto \mu_t$ we conclude that $\mu_t=0$ for all $t\in [0,1]$.
\end{proof}

\begin{remark}\label{rem:FP_div_free}
    When $\nabla\cdot u=0$, by suitable approximation arguments, one can show that the stochastic flow $X^\kappa_{s,t}(x,\omega)$ from Theorems~\ref{thm:SWP_fBm}-\ref{thm:SWP_stable} is Lebesgue measure preserving, in the sense that there exists $\Gamma\subset\Omega$ of full probability such that
    \begin{equation*}
        \int_{\T^d} f(X^\kappa_{s,t}(x,\omega)) \dd x = \int_{\T^d} f(x) \dd x\quad \forall\, f\in L^1(\T^d),\quad \forall \, \omega\in\Gamma,\quad\forall\, 0\leq s\leq t\leq T.
    \end{equation*}
    As a consequence, in the setting of Theorem~\ref{thm:SWP_stable}, if $\mu_0({\rm d} x)=\rho_0(x) \dd x$ for some $\rho_0\in L^p(\T^d)$, $p\in [1,\infty]$, then the unique solution $\mu^\kappa$ to~\eqref{eq:FP_PDE_stable} (which can now be regarded as an advection-diffusion equation with fractional viscosity) is of the form $\mu^\kappa_t({\rm d} x)=\rho^\kappa_t(x)\dd x$, where $\rho^\kappa\in C([0,1];L^p(\T^d))$ with $\| \rho^\kappa_t\|_{L^p(\T^d)}\leq \| \rho_0\|_{L^p(\T^d)}$.
    In particular, $\rho^\kappa$ is now a weak, $L^p$-valued solution to the PDE
    \begin{equation}\label{eq:FP_PDE_stable_v2}
        \partial_t \rho^\kappa + u\cdot\nabla \rho^\kappa = - \kappa^\beta |\nabla|^\beta \rho^\kappa,\quad\rho^\kappa\vert_{t=0}=\rho_0.
    \end{equation}
    Furthermore, since $X^\kappa$ is Lebesgue measure preserving and invertible, after a change of variables, formula~\eqref{eq:FP_representation_formula1} yields the pointwise representation formula à la Feynman--Kac
    \begin{equation}\label{eq:FP_representation_formula2}
        \rho^\kappa_t(x)=\E\big[\rho_0\big((X^\kappa_{0,t})^{-1}(x)\big)\big].
    \end{equation}
\end{remark}

\begin{remark}
    Theorem~\ref{thm:SWP_stable} and Remark~\ref{rem:FP_div_free} were stated for $\beta\in (0,2)$, in which case $W^\beta$ is a pure jump process, but the same results apply verbatim for $\beta=2$, in which case $W^\beta$ is just standard Brownian motion.
    Let us also mention that in this work we restrict ourselves to isotropic $\beta$-stable processes only for the sake of simplicity. Results in the style of Theorem~\ref{thm:SWP_stable} are valid for a much larger class of stable-like processes, like the ones satisfying the abstract assumptions from the works~\cite{CSZ2018,BDG2025}.
\end{remark}

In light of formula~\eqref{eq:FP_representation_formula2}, one can derive a \emph{Fluctuation-Dissipation Relation} (FDR) for the PDE~\eqref{eq:FP_PDE_stable_v2}, describing the energy dissipated up to time $t$ in terms of quantities associated to the underlying (backward) stochastic flow. This is a natural generalization to $\beta\in (0,2)$ of~\cite[eq. (2.12)]{DriEyi2017}, where the Brownian case $\beta=2$ was treated.\footnote{Therein $\kappa^\beta$ is replaced by $\kappa$, simply because of the different definition of the coefficient in front of $W^\beta$ in the SDE~\eqref{eq:stable_SDE}.}

\begin{corollary}[Fluctuation-Dissipation Relation]\label{cor:FDR}
    Let $W^\beta$, $u$ be as in Theorem~\ref{thm:SWP_stable} and additionally assume that $u$ is divergence-free. Then for any $\rho_0\in L^2$ and any $\kappa>0$, the associated solution $\rho^\kappa$ to the PDE~\eqref{eq:FP_PDE_stable_v2} satisfies
    \begin{equation}\label{eq:fluctuation_dissipation}
        \| \rho_0\|_{L^2}^2-\|\rho^\kappa_t\|_{L^2}^2 = 2\kappa ^\beta\int_0^t \big\| |\nabla|^\beta \rho^\kappa_s\big\|_{L^2}^2 \dd s = \int_{\T^d} Var\big[ \rho_0((X^\kappa_{0,t})^{-1}(x))\big] \dd x
    \end{equation}
\end{corollary}

\begin{proof}
For simplicity, we take $\rho_0$ and $u$ to be smooth; the general case then follows from approximation arguments (the stochastic flows $X^\kappa$ from Theorem~\ref{thm:SWP_stable} are also well-behaved under approximations of the drift $u$, cf.~\cite{CSZ2018,BDG2025}). As a consequence, we may assume the solution $\rho^\kappa$ to be smooth as well.
The first equality in~\eqref{eq:fluctuation_dissipation} follows from energy balance, obtained by testing $\rho^\kappa$ against the PDE~\eqref{eq:FP_PDE_stable_v2} itself.
Concerning the second equality, since $X^\kappa_{0,t}$ and its inverse are Lebesgue measure-preserving by Remark~\ref{rem:FP_div_free}, we have
\begin{align*}
    \| \rho_0\|_{L^2}^2
    = \int_{\T^d} |\rho_0(x)|^2 \dd x
    = \E\left[ \int_{\T^d} |(\rho_0(X^\kappa_{0,t})^{-1}(x))|^2 \dd x \right]
    = \int_{\T^d} \E\big[|(\rho_0(X^\kappa_{0,t})^{-1}(x))|^2\big] \dd x.
\end{align*}
On the other hand, by~\eqref{eq:FP_representation_formula2} it holds that
\begin{align*}
    \| \rho_t^\kappa\|_{L^2}^2
    = \int_{\T^d} |\rho_t^\kappa(x)|^2 \dd x
    = \int_{\T^d} \big| \E[(\rho_0(X^\kappa_{0,t})^{-1}(x))]\big|^2 \dd x.
\end{align*}
Combining the two above identities yields the desired conclusion.
\end{proof}

\begin{remark}\label{rem:regbynoise_time_integrability}
    In Theorems~\ref{thm:SWP_fBm}-\ref{thm:SWP_stable} we focused on $u\in L^\infty([0,1];C^\alpha(\T^d))$, as this is the case mostly studied in the literature, but similar results holds for drifts $u\in L^q([0,1];C^\alpha(\T^d))$ with $q<+\infty$ large enough, up to minor modifications. For fBm, the results from~\cite{GalGer2025} apply for
    \begin{equation*}
        u\in L^q([0,1];C^\alpha(\T^d)),\quad q\geq 2,\quad \alpha>1-\frac{1}{2H};
    \end{equation*}
    instead for isotropic $\beta$-stable processes, the results from~\cite{tian2025sdes} guarantee strong well-posedness under the condition
    \begin{equation*}
        u\in L^q([0,1];C^\alpha(\T^d)), \quad \alpha>1-\frac{\beta}{2},\quad \alpha>1-\beta+\frac{\beta}{q}.
    \end{equation*}
    Notice how one can handle simultaneously a special subcase of both results by enforcing the condition
    \begin{equation}\label{eq:condition_SWP_time_integrable}
        u\in L^q([0,1];C^\alpha(\T^d)), \quad q\geq 2,\quad \alpha>\alpha_W.
    \end{equation}
\end{remark}

\section{Stochastic estimates for oscillatory integrals}\label{sec:stochastic_estimates}

The goal of this section is to state and prove Lemma~\ref{lem:lucio_v2} below, capturing the smallness (in probability) of integrals of the form $\int_s^t f(r,X_r) \dd r$, due to cancellations occurring whenever $f$ is concentrated on high frequencies (thus highly oscillatory) and the stochastic process $X$ is ``irregular'' enough.
The oscillatory nature of $f$ is measured via negative norms $\| \cdot\|_{C^{-\theta}_x}$, while the properties of $X$ are encoded in it being either a fractional Brownian motion or a stable process, up to a more-regular-in-time perturbation (see condition~\eqref{eq:slowly_varying_process} below).

To this end, we preliminarily need to introduce a few notations and recall a version of L\^e's Stochastic Sewing Lemma (SSL)~\cite{le2020sewing}, specifically the SSL with shifts from~\cite{Gerencser2023}.

Given an interval $[S,T]$, we denote the associated 2- and 3-dimensional simplices by $[S,T]^2_\leq :=\{(s,t)\in [S,T]^2 : s\leq t\}$, $[S,T]^3_\leq :=\{(s,u,t)\in [S,T]^3 : s\leq u\leq t\}$. For $(s,t)\in [S,T]^2_\leq$, define $s_-:=s-(t-s)$. We then define slightly more restricted sets of pairs/triples by
\begin{align*}
    & \overline{[S,T]}^2_\leq :=\{(s,t)\in [S,T]^2_\leq : s_-\geq S\},\\
    & \overline{[S,T]}^3_\leq :=\{(s,u,t)\in [S,T]^3_\leq : \min\{(u-s),(t-s)\geq (t-s)/3, \ s_-\geq S\}.
\end{align*}
The next statement is a simplified version of~\cite[Lem. 2.2]{Gerencser2023} (see also~\cite[Lem. 2.5]{GalGer2025} for a more general variant). Therein we fix a probability space $(\Omega,\cF,\P)$ on which we are given a generic filtration $(\cF_t)_{t\in [S,T]}$. Recall the notation $\E_s = \E[\,\cdot\, | \cF_s]$.

\begin{lemma}[Shifted SSL]\label{lem:shifted_SSL}
    Let $(A_{s,t}: (s,t)\in \overline{[S,T]}^2_\leq)$ be a family of random variables in $L^2(\Omega)$ such that $A_{s,t}$ is $\cF_t$-measurable. Suppose that there exist positive constants $C_1$, $C_2$, $\eps_1$, $\eps_2$ such that
    \begin{align}
        \| A_{s,t}\|_{L^2(\Omega)} & \leq C_1 |t-s|^{\frac{1}{2}+\eps_1} & \forall\, (s,t)\in \overline{[S,T]}^2_\leq, \label{eq:shifted_SSL_assumption1}\\
        \| \E_{s-(t-s)} (A_{s,t}-A_{s,u}-A_{u,v})\|_{L^2(\Omega)} & \leq C_2 |t-s|^{1+\eps_2} & \forall\, (s,u,t)\in \overline{[S,T]}^3_\leq.\label{eq:shifted_SSL_assumption2}
    \end{align}
    Then there exist finite constants $K_1$, $K_2$, only depending on $\eps_1$, $\eps_2$, and a unique (up to modification) $(\cF_t)_t$-adapted process $\mathcal{A}$ such that $\mathcal{A}_S=0$ and the following bounds hold:
    \begin{align}
        \| \mathcal{A}_t-\mathcal{A}_s- A_{s,t}\|_{L^2(\Omega)} & \leq K_1 C_1 |t-s|^{\frac{1}{2}+\eps_1} + K_2 C_2 |t-s|^{1+\eps_2} & \forall\, (s,t)\in \overline{[S,T]}^2_\leq, \label{eq:shifted_SSL_conclusion1}\\
        \| \E_{s-(t-s)} (\mathcal{A}_t-\mathcal{A}_s- A_{s,t})\|_{L^2(\Omega)} & \leq K_2 C_2 |t-s|^{1+\eps_2} & \forall\, (s,u,t)\in \overline{[S,T]}^2_\leq,\label{eq:shifted_SSL_conclusion2}\\
        \| \mathcal{A}_t-\mathcal{A}_s\|_{L^2(\Omega)} & \leq K_1 C_1 |t-s|^{\frac{1}{2}+\eps_1} + K_2 C_2 |t-s|^{1+\eps_2} & \forall\, (s,t)\in [S,T]^2_\leq \label{eq:shifted_SSL_conclusion3}.
    \end{align}
    Moreover the increments $\mathcal{A}_t-\mathcal{A}_s$ are the unique limits in $L^2(\Omega)$ of Riemann sums over dyadic partition of $A_{s,t}$ (see~\cite[eq. (2.9)]{GalGer2025} for the exact definition).
\end{lemma}

In the same setting as above, let $\psi:[S,T]\times \Omega\to \R$ be a $(\cF_t)_{t}$-progressive process. 
Given $\varsigma>0$, we denote by $\llbracket \psi\rrbracket_\varsigma$ the smallest deterministic constant such that
\begin{equation}\label{eq:slowly_varying_process}
    \| \E_s|\psi_t-\E_s\psi_t|\|_{L^\infty(\Omega)}  \leq \llbracket \psi\rrbracket_\varsigma |t-s|^\varsigma \quad\forall\,
    s,t\in [S,T]^2_\leq.
\end{equation}
This condition quantifies the (conditional) error made by trying to predict $\psi_t$ given the available information $\cF_s$.
In the next statement, given a reference filtration $(\cF_t)_t$, $W$ is either a $(\cF_t)_t$-fBm or a $(\cF_t)_t$-$\beta$-stable process, in the sense discussed respectively in Sections~\ref{subsec:fbm}-\ref{subsec:stable_levy}.

\begin{lemma}\label{lem:lucio_v2}
Let $(\cF_t)_{t\in [S,T]}$ be a given filtration, $W^H$ be a $\R^d$-valued $(\cF_t)_t$-fBm of parameter $H\in (0,+\infty)$ and let $f\in L^\infty([S,T];C^1_x)$. Let $Z$ be a $\R^d$-valued $\cF_S$-measurable random variable and $\psi$ be a $\R^d$-valued, $(\cF_t)_t$-progressive process such that $\llbracket \psi\rrbracket_\varsigma<+\infty$ for some $\varsigma>0$ (for $\llbracket \psi\rrbracket_\varsigma$ defined by~\eqref{eq:slowly_varying_process}). Then:
\begin{enumerate}
    \item If $H\in (0,1/2)$ and $\varsigma=1$, there exists a constant $C_1=C_1(H,d)>0$ such that, uniformly over $\kappa\in (0,1]$, for every $(s,t)\in [S,T]^2_\leq$ one has
    \begin{equation}\label{eq:sewing_fBm1}
	   \bigg\|\int_s^t f(r, Z+ \kappa W^H_r + \kappa\psi_r) \dd r \bigg\|_{L^2(\Omega)} \leq C_1 \kappa^{-1} |t-s|^{1-H} (1+\llbracket \psi\rrbracket_1)\| f\|_{L^\infty([S,T];C^{-1}_x)}.
    \end{equation}
    \item If $H\in [1/2,+\infty)\setminus \N$ and $\varsigma>H$, then for any $\theta>0$ such that
    \begin{equation}\label{eq:SSL_regularity_parameter_fbm}
        \theta<\frac{1}{2H}, \quad \theta<\frac{\varsigma}{H}-1
    \end{equation}
    there exists a constant $C_2=C_2(H,d,\theta,\varsigma)>0$ such that, uniformly over $\kappa\in (0,1]$, for every $(s,t)\in [S,T]^2_\leq$ one has
    \begin{equation}\label{eq:sewing_fBm2}
	   \bigg\| \int_s^t f(r, Z+ \kappa W^H_r + \kappa \psi_r) \dd r\bigg\|_{L^2(\Omega)} \leq C_2 \kappa^{-\theta} |t-s|^{1-\theta H} (1+\llbracket \psi\rrbracket_\varsigma)\| f\|_{L^\infty([S,T];C^{-\theta}_x)}.
    \end{equation}
\end{enumerate}
Similarly, if $f$, $Z$, $\psi$ are as above and $W^\beta$ is a $\beta$-stable Lévy process with $\beta\in (0,2]$ and $\varsigma>1/\beta$, then for any $\theta>0$ such that
\begin{equation}\label{eq:SSL_regularity_parameter_stable}
        \theta<\frac{\beta}{2}, \quad \theta<\varsigma\beta-1
    \end{equation}
there exists a constant $C_3=C_3(\beta,d,\theta,\varsigma)>0$ such that, uniformly over $\kappa\in (0,1]$, for every $(s,t)\in [S,T]^2_\leq$
    \begin{equation}\label{eq:sewing_stable}
	   \bigg\| \int_s^t f(r, Z+ \kappa W^\beta_r + \kappa \psi_r) \dd r\bigg\|_{L^2(\Omega)}
       \leq C_3 \kappa^{-\theta} |t-s|^{1-\theta H} (1+\llbracket \psi\rrbracket_\varsigma)\| f\|_{L^\infty([S,T];C^{-\theta}_x)}.
    \end{equation}
\end{lemma}

\begin{proof}
    The proof follows similar computations as~\cite[Lem. 3.1]{GalGer2025} and is based on an application of Lemma~\ref{lem:shifted_SSL} to
    \begin{equation*}
        A_{s,t}:=\E_{s-(t-s)} \int_s^t f(r,Z+\kappa W_r + \kappa \E_{s-(t-s)} \psi_r) \dd r;
    \end{equation*}
    arguing as in~\cite{GalGer2025}, using that $f\in L^\infty(S,T;C^1_x)$, it's easy to see that $\mathcal{A}_t-\mathcal{A}_s=\int_s^t f(r,Z+\kappa W_r+\kappa\psi_r)\dd r$ is the unique limit of the Riemann sums associated to $A_{s,t}$; therefore, in order to deduce the estimates~\eqref{eq:sewing_fBm1}-\eqref{eq:sewing_fBm2}-\eqref{eq:sewing_stable}, it suffices to verify that conditions~\eqref{eq:shifted_SSL_assumption1}-\eqref{eq:shifted_SSL_assumption2} are met (for suitable constants $C_i$) and then apply~\eqref{eq:shifted_SSL_conclusion3} (keeping in mind that $|t-s|\leq |T-S|\leq 1$).

    First consider the case of fBm $W^H$ with $H\in (0,+\infty)\setminus\N$.  For notational convenience, let us set $f^\kappa_r(x):=f(r,Z+\kappa x)$.\footnote{Technically this makes $f^\kappa_r$ a random function due to the presence of $Z$. However, since $Z$ is $\cF_S$-adapted and we are always conditioning on $\cF_v$ with $v\geq S$, we can effectively treat it as a deterministic constant.}
    By the LND property (cf.~\eqref{eq:LND_fBm}), we have
    \begin{align*}
        A_{s,t}
        =\int_s^t \E_{s-(t-s)} f(r,Z+\kappa W^H_r + \kappa \psi_r) \dd r
        = \int_s^t (\mathcal{P}_{c_H |r-s+t-s|^{2H}} f^\kappa_r)(\E_{s-(t-s)} (W_r+\psi_r)) \dd r;
    \end{align*}
    heat kernel estimates and the properties of Besov norms under translations and rescalings then imply that
    \begin{align*}
        \| A_{s,t}\|_{L^2(\Omega)}\leq \int_s^t \| \mathcal{P}_{c_H |r-s+t-s|^{2H}} f^\kappa_r\|_{C^0_x} \dd r
        & \lesssim |t-s|^{-\theta H} \int_s^t \| f^\kappa_r\|_{C^{-\theta}_x} \dd r\\
        & \lesssim |t-s|^{1-\theta H} \kappa^{-\theta} \| f\|_{L^\infty([S,T];C^{-\theta}_x)}.
    \end{align*}
    Now fix $(s,u,t)\in \overline{[S,T]}^3_\leq$ and let $s_1=s-(t-s)$, $s_2=s-(u-s)$, $s_3=u-(t-u)$, $s_4=s$, $s_5=u$, $s_6=t$.
    These points are almost ordered according to their indices, except $s_3$ and $s_4$, for which $s_4\leq s_3$ may happen, but this does not affect the computations below.
    As in~\cite[Lem. 3.1]{GalGer2025}, one has
    \begin{align*}
        \E_{s-(t-s)} (A_{s,t}-A_{s,u}-A_{u,v})
        = I+J
        := & \E_{s_1} \E_{s_2} \int_{s_4}^{s_5} [f^\kappa_r(W^H_r + \E_{s_1} \psi_r)-f^\kappa_r(W^H_r + \E_{s_2} \psi_r)] \dd r\\
        & + \E_{s_1} \E_{s_3} \int_{s_5}^{s_6} [f^\kappa_r(W^H_r + \E_{s_1} \psi_r)-f^\kappa_r(W^H_r + \E_{s_3} \psi_r)] \dd r;
    \end{align*}
    the estimates for the two terms are similar, so let us focus on $I$. Again thanks to the LND property~\eqref{eq:LND_fBm} and heat kernel estimates, it holds that
    \begin{align*}
        |I|
        & \leq \E_{s_1} \int_{s_4}^{s_5} |\mathcal{P}_{c_H|r-s_2|^{2H}} f^\kappa_r(\E_{s_2} W^H_r + \E_{s_1} \psi_r)-\mathcal{P}_{c_H|r-s_2|^{2H}} f^\kappa_r(\E_{s_2} W^H_r + \E_{s_2} \psi_r)|\\
        & \leq \E_{s_1} \int_{s_4}^{s_5} \| \mathcal{P}_{c_H|r-s_2|^{2H}} f^\kappa_r\|_{C^1_x} |\E_{s_1} \psi_r-\E_{s_2} \psi_r| \dd r\\
        & \lesssim \int_{s_4}^{s_5} |s_4-s_2|^{-(1+\theta) H} \| f^\kappa_r\|_{C^{-\theta}_x} \E_{s_1}|\E_{s_1} \psi_r-\E_{s_2} \psi_r| \dd r\\
        & \lesssim \kappa^{-\theta}\| f\|_{L^\infty([S,T];C^{-\theta}_x)} |t-s|^{-(1+\theta) H} \int_{s_4}^{s_5} \E_{s_1}|\E_{s_1} \psi_r- \psi_r| \dd r\\
        & \lesssim \kappa^{-\theta}\| f\|_{L^\infty([S,T];C^{-\theta}_x)} \llbracket \psi\rrbracket_\varsigma |t-s|^{1+\varsigma-(1+\theta) H}
    \end{align*}
    where in the above we used several times properties of conditional expectation (e.g. $\E_{s_2} \E_{s_1} = \E_{s_1}$) and the assumption $(s,u,t)\in \overline{[S,T]}^3_\leq$.
    Running a similar estimate for $J$ and taking the $L^2(\Omega)$-norm on both sides, overall we find
    \begin{equation}\label{eq:SSL_proof_main_estim}\begin{split}
        & \| A_{s,t}\|_{L^2(\Omega)} \lesssim \kappa^{-\theta} \| f\|_{L^\infty([S,T];C^{-\theta}_x)} |t-s|^{1-\theta H},\\
        & \| \E_{s-(t-s)} (A_{s,t}-A_{s,u}-A_{u,v})\|_{L^2(\Omega)} \lesssim \kappa^{-\theta}\| f\|_{L^\infty([S,T];C^{-\theta}_x)} \llbracket \psi\rrbracket_\varsigma |t-s|^{1+\varsigma-(1+\theta) H}.
    \end{split}\end{equation}
    This implies verification of~\eqref{eq:shifted_SSL_assumption1}-\eqref{eq:shifted_SSL_assumption2} with
    \begin{equation*}
        C_1\sim \kappa^{-\theta} \| f\|_{L^\infty([S,T];C^{-\theta}_x)}, \quad
        C_2\sim \kappa^{-\theta}\| f\|_{L^\infty([S,T];C^{-\theta}_x)} \llbracket \psi\rrbracket_\varsigma
    \end{equation*}
    if the parameters $\theta,\varsigma$ are such that
    \begin{equation*}
        1-\theta H >\frac{1}{2}, \quad 1+\varsigma-(1+\theta) H > 1
    \end{equation*}
    which are exactly equivalent to~\eqref{eq:SSL_regularity_parameter_fbm}.
    Therefore~\eqref{eq:shifted_SSL_conclusion3} holds; taking into account $\varsigma>H$ and $|t-s|\leq 1$, this implies~\eqref{eq:sewing_fBm2}.

    When $H<1/2$ and $\varsigma=1$, then condition~\eqref{eq:SSL_regularity_parameter_fbm} is satisfied by $\theta=1$, therefore yielding~\eqref{eq:sewing_fBm1}.

    The proof in the $\beta$-stable case is almost identical, up to applying the independence of increments property instead of LND (cf.~\eqref{eq:LND_stable}) and smoothing estimates in Besov spaces associated to $\mathcal{P}^H$ (cf. Lemma~\ref{lem:stable_kernel_estimates}). In particular, one has
    \begin{equation*}
        \| \mathcal{P}^\beta_{r-s} f^\kappa_r\|_{C^0_x} \lesssim \kappa^{-\theta} \| f\|_{L^\infty([S,T];C^{-\theta}_x)} |r-s|^{-\frac{\theta}{\beta}},\quad
        \| \mathcal{P}^\beta_{r-s} f^\kappa_r\|_{C^1_x} \lesssim \kappa^{-\theta} \| f\|_{L^\infty([S,T];C^{-\theta}_x)} |r-s|^{-\frac{1+\theta}{\beta}}
    \end{equation*}
    and so similar computations yield
    \begin{align*}
        & \| A_{s,t}\|_{L^2(\Omega)} \lesssim \kappa^{-\theta} \| f\|_{L^\infty([S,T];C^{-\theta}_x)} |t-s|^{1-\frac{\theta}{\beta}},\\
        & \| \E_{s-(t-s)} (A_{s,t}-A_{s,u}-A_{u,v})\|_{L^2(\Omega)} \lesssim \kappa^{-\theta}\| f\|_{L^\infty([S,T];C^{-\theta}_x)} \llbracket \psi\rrbracket_\varsigma |t-s|^{1+\varsigma-\frac{1+\theta}{\beta}},
    \end{align*}
    which again allows to apply Lemma~\ref{lem:shifted_SSL} under condition~\eqref{eq:SSL_regularity_parameter_stable} to deduce estimate~\eqref{eq:sewing_stable}.
\end{proof}

\begin{remark}\label{rem:practical_applications_SSL}
    To help the reader, anticipating what is to come, let us explain how Lemma~\ref{lem:lucio_v2} will be applied concretely:
    \begin{itemize}
        \item In the proof of our main result, Theorem~\ref{th:main_theorem}, in Section~\ref{sec:proof_main_thm}, we will take $\psi\equiv 0$. Then $\varsigma$ can be taken arbitrarily large and condition~\eqref{eq:SSL_regularity_parameter_fbm} (respectively~\eqref{eq:SSL_regularity_parameter_stable}) reduces to $\theta<\frac{1}{2H}$ (resp. $\theta<\frac{\beta}{2}$).
        \item In the construction of an autonomous vector field for $d\geq 3$, presented in Section~\ref{subsec:autonomous}, we will apply Lemma~\ref{lem:lucio_v2} for a non-trivial choice of $\psi=\psi^\kappa$; in this case, we have a uniform estimate on $\llbracket \psi^\kappa \rrbracket_\varsigma$ for $\varsigma =1+\alpha H$ (resp. $\varsigma =1+\frac{\alpha}{\beta}$) when the noise is fBm (resp. $\beta$-stable), allowing to verify that conditions~\eqref{eq:SSL_regularity_parameter_fbm}-\eqref{eq:SSL_regularity_parameter_stable} are still met.    
    \end{itemize}
\end{remark}

\section{The vector field $u$}
\label{sec:construction}

In this section, we construct the deterministic velocity field  $u$ which plays a key role in the proof of Theorem~\ref{th:main_theorem}. The construction is inspired by~\cite{CCS2023}.

\subsection{Choice of parameters}\label{subsec:parameters}

We aim to define a super-exponentially decreasing sequence of parameters 
$$a_{q+1} \approx a_q^{1+ \delta}$$
for some $\delta \in (0, 1/100)$ and $a_0 \in (0,1)$. Since we want to deal with natural numbers, we consider $a_0 \in (0,1)$ and we define $(a_q)_{q\in\N}$ so that
\begin{align} \label{eq:a-q}
     a_q/a_{q+1} \in [a_q^{-\delta}, a_q^{-\delta} + 4 ] \cap 4\N \,, \qquad 1/a_q \in 2\N, \quad \forall q \geq 0\,.
\end{align}
We also define the mollification parameter $\ell_q$, the noise parameter $\kappa_q$ and the time parameter $t_{q} \uparrow 1/2$ with $t_0 =0$ as follows:
\begin{align} 
\ell_q:=  a_q^{\gamma} a_{q+1}^{1-\gamma} \quad \forall q\in\N, \label{eq:ell-q}
    \\
    \kappa_q := a_q^{1/2} a_{q+1}^{1/2}  \quad \forall q\in\N, \label{eq:kappa-q}
    \\
    t_{q+1} := t_q + \frac{6}{\pi^2 (q+1)^2} \quad \forall q\in\N,  \label{eq:t-q}
\end{align}
where $\gamma \in (1/2, 1)$ is chosen depending on the process $W$ under consideration.
More precisely,  we can impose that:
\begin{enumerate}[label=(\roman*), ref=\roman*]
    \item \label{eq:i} $\frac{1}{2} < \gamma < 1$ when $W$ is  a fractional Brownian motion of parameter $H \in (0, 1/2)$;
    \item \label{eq:ii} $\max \{ 1- \frac{1}{4H}  , \frac 1 2 \} < \gamma  < 1$ when $H \in [1/2, +\infty)$;
    \item \label{eq:iii} $\max \{ 1- \frac{\beta}{4} , \frac 1 2 \} < \gamma < 1$ when $W$ is a $\beta$-stable Lévy process with $\beta \in (0, 2]$.
\end{enumerate}
Finally, we choose $a_0 \in (0,1)$ sufficiently small in terms of $\delta>0$ and $\gamma\in (1/2,1)$ so that 
\begin{align} \label{eq:a-0}
    \sum_{q=j}^\infty a_q^{\delta(1-\gamma)} < 4a_j^{\delta (1- \gamma)} \quad\forall\,j\geq 0.
\end{align}
The latter follows from
$$
\sum\limits_{q =j}^\infty a_q^{\delta(1-\gamma)} \leq 2 \sum\limits_{q =j}^\infty a_j^{\delta(1-\gamma)(1+\delta)^{q-j}} \leq 2 \sum\limits_{q =0}^\infty  a_j^{\delta(1-\gamma)(1+ \delta q)} =  \frac{2a_j^{\delta(1-\gamma)}}{1-a_j^{\delta^2(1-\gamma)}} < 4 a_j^{\delta (1-\gamma)},
$$
where we used that $a_q\leq a_j$; the last inequality follows from $a_j \leq a_0$, as long as we choose $a_0$ small enough so that $a_0^{\delta^2(1-\gamma)}<1/2$.
%
Since $\gamma>1/2$, up to choosing $a_0$ even smaller, we can further assume that
\begin{align} \label{eq:kappa-q-ell-q}
     10 q^4 \kappa_q \leq \ell_q  \qquad \forall q \in \N.
\end{align}
\begin{convention}\label{convention_parameters}
    From now on, we assume throughout the paper that the parameters $\gamma$, $a_0$, $\delta$ and $(a_q)_q$ are fixed and satisfy conditions~\eqref{eq:a-q}-\eqref{eq:a-0}-\eqref{eq:kappa-q-ell-q}. Whenever needed, we allow ourselves to take $a_0$ even smaller, while still satisfying the same constraints, so to reabsorb irrelevant constants in the estimates.
\end{convention}

Finally, we define the odd and even chessboard sets of size $a>0$ with $a^{-1} \in 2 \mathbb{N}$. Let
\[
\pi(k) := k_1 + k_2 \pmod{2}, \quad k \in \Z^2,
\]
then we define the even and the odd chessboards of size $a$
\begin{equation}\label{eq:defn_chessboards}
    \mathcal{C}^{\even}_a :=
    \Bigg( \bigcup_{\substack{k \in \mathbb{Z}^2\\ \pi(k)=0}} ak + [0,a)^2  \Bigg) \cap [0,1)^2 \qquad
    \mathcal{C}^{\mathrm{odd}}_a  :=
    \Bigg( \bigcup_{\substack{k \in \mathbb{Z}^2\\ \pi(k)=1}} ak + [0,a)^2 \Bigg) \cap [0,1)^2 \,.
\end{equation}

\subsection{Construction of the velocity field}\label{subsec:construction_velocity}

Recall that $ 1/a_q \in 2\N$ for any $q \in \N$; therefore to define a function on $\T$, we can first define it in $[0, 2a_q]$ and then extend by periodicity. Let us split $(0, 2a_q]$ into the following intervals:
\[
I_{i}= \left( 2ia_{q+1}, 2(i+1)a_{q+1}\right], \quad \text{ for } i \in \{0, 1, \dots, N_q\}, \quad N_q:=\frac{a_q}{a_{q+1}}-1.
\]
We specify $\varphi_{q+1}$ on $[0,2a_q]$ as follows: we set $\varphi_{q+1} (0) =0$; for all the intervals $I_i \subset (0, a_q]$, we define $\varphi_{q+1}:I_i \to \R$ to be piecewise linear with derivative constantly equal $+1$ in the first half of $I_i$, and $-1$ in the second half of it. By construction, $\varphi_{q+1}$ is continuous on $[0, 2a_q]$ with $\varphi_{q+1}(0)=\varphi_{q+1}(2a_q)=0$.

This implies that $\sup_{t \in I_i} |\varphi_{q+1}(t)| =  a_{q+1}$. Then we extend $\varphi_{q+1}$ to the interval $(a_q,2a_q]$ as an odd extension:
$$ {\varphi_{q+1}}\big\vert_{(a_q,2a_q]} :=  -{\varphi_{q+1}}\big\vert_{(0, a_q]}(\cdot -a_q) \,.$$
Finally, we extend $\varphi_{q+1}$ by $2a_q$-periodicity, and consider the 1-periodic map $\varphi_{q+1} : \T \to \R$. 
From the construction it is clear that
\begin{equation}\label{eq:estim_varphiq}
    \norm{\varphi_{q+1}}_{C^0} = a_{q+1}, \quad \norm{\varphi_{q+1}'}_{C^0} = 1.
\end{equation}

Let $v_{q+1,1}:\T^2 \to \R^2$ be the divergence-free vector field (horizontal shear flow) given by:
\begin{equation}
v_{q+1, 1}(x_1, x_2):= \frac{a_q}{2}\begin{bmatrix}
           \varphi_{q+1}'(x_2) \\
           0
         \end{bmatrix}.
\end{equation}
It follows from~\eqref{eq:estim_varphiq} that
\begin{align} 
    \norm{v_{q+1,1}}_{C^{-1}_x} \lesssim \frac{a_q}{2}\norm{\varphi_{q+1}}_{C^0_x} =\frac{a_{q+1} a_q }{2}, \quad 
    \|v_{q+1,1}\|_{C^0_x} = \frac{a_q}{2};
\end{align}
by interpolation inequalities, we then also have 
\begin{align} \label{eq:norm-1}
    \|v_{q+1,1}\|_{C^{-\theta}_x}\lesssim a_{q+1}^\theta a_q\quad\forall\, \theta\in [0,1].
\end{align}
We similarly define a vertical shear flow $v_{q+1,2 }: \T^2 \to \R^2$ by:
\begin{equation}
v_{q+1, 2}(x_1, x_2):= \frac{a_{q+1}}{2}\begin{pmatrix}
           0 \\
           \varphi_{q+1}'(x_1) + 1
         \end{pmatrix}.
\end{equation}
Using $\|v_{q+1, 2} \|_{C^0_x} \leq a_{q+1}$ and interpolation inequalities, we similarly obtain
\begin{align} \label{eq:norm-2}
    \|v_{q+1,2}\|_{C^{-\theta}_x}\lesssim a_{q+1}^{1+\theta} \quad\forall\, \theta\in [0,1].
\end{align}

\begin{figure}[htp!]
\centering
\includegraphics[width=0.85\linewidth]{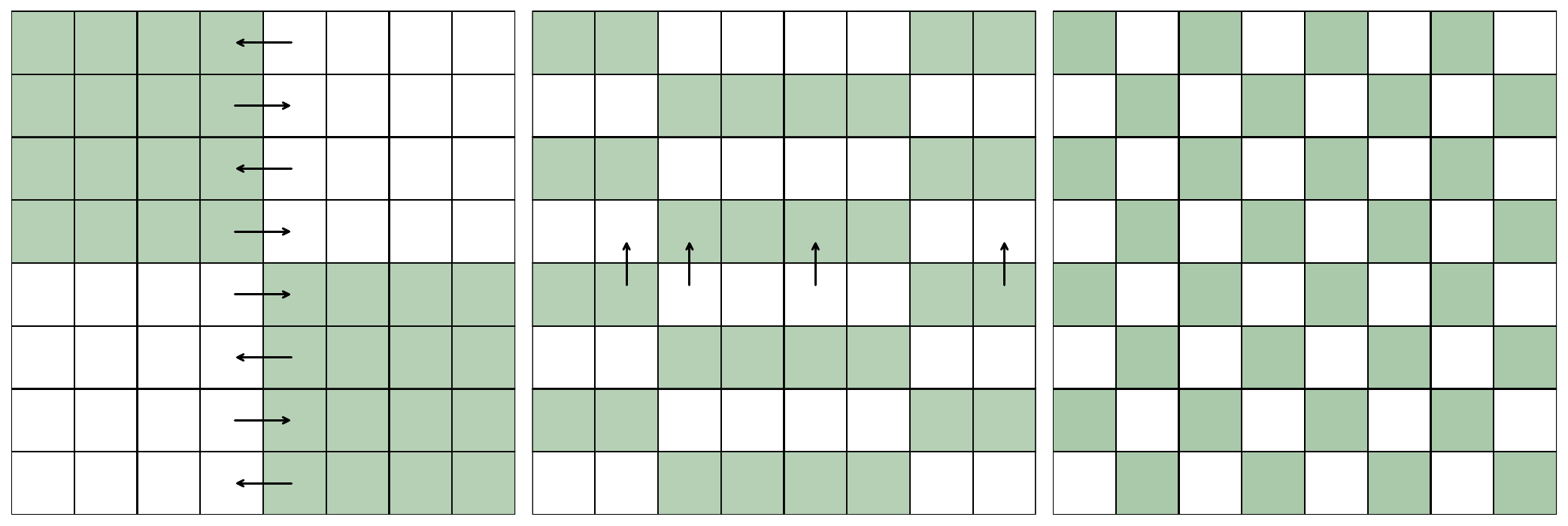}
\caption{The figure on the left shows the even chessboard set at scale $a_0$. The figure in the center shows its image under the time-$1$ Lagrangian flow generated by the vector field $v_{1,1}$. The figure on the right shows the image obtained by then applying the time-$1$ Lagrangian flow generated by $v_{1,2}$ to the previous image. The arrows indicate the velocity field $v_{1,1}$ in the center figure and $v_{1,2}$ in the right figure.}
\label{fig:chessboards}
\end{figure}

We define $ u_{\text{swap}}:\T^2 \to \R^2$ to be the constant vector field
$$v_{\text{swap}, q} := \begin{pmatrix}a_{q+1} \\ 0\end{pmatrix} \,.$$

Recall the definition of $\{t_q\}_{q\in\N}$ from~\eqref{eq:t-q}. Let us partition the intervals $[t_q,t_{q+1})$ by introducing
\begin{equation}\label{eq:defn_J_qi}
    s_q:=t_q + \frac{t_{q+1} - t_q}{3},\quad
    J_{q,1} :=  \left [ t_q, s_q \right ), \quad
    J_{q,2} := J_{q,1} + \frac{t_{q+1} - t_q}{3}.
\end{equation}

For $i =1,2,3$, let us take some smooth time cut-offs $\eta_{q, i} \in C^\infty_c (J_{q,i};[0, +\infty))$ satisfying
\begin{align}
\label{eq:eta-time-mollifier}
    \int_{J_{q,i}} \eta_{q,i} =1 \,, \quad \| \eta_{q,i} \|_{C^k(J_{q,i})} \lesssim_k (t_{q+1} - t_q)^{-k-1} \lesssim_k q^{2 (k+1)} \quad \forall k \in \N.
\end{align}

Let $\chi: \R \to [0, 1]$ be a standard mollifier and let  $ \chi_{\ell_{q+1}} (\cdot ) =\frac{1}{\ell_{q+1}}\chi(\frac{\cdot}{\ell_{q+1}})$ \footnote{If $f:\T \to \R$ and $\chi:\R \to [0,1]$ is a mollifier, one considers the $1$-periodic extension of $f$ to $\R$, defines $f \star \chi : \R \to \R$ in the classical sense, and notices that it is $1$-periodic.}.

With all the above preparations, we can define $u$ as follows.
For any $q \in \N$ and   $t \in [t_q, t_{q+1})$ we set
\begin{equation}\label{eq:def_u_half}
{u}(t, x_1, x_2):=\begin{cases}
    \tilde v_{q,1}(t,x_2):=\eta_{q,1}(t) \left( \chi_{\ell_{q+1}} \star v_{q+1,1}\right)(x_2) & \quad \text{ if }t \in J_{q,1}, \\
    \tilde v_{q,2}(t,x_1):=\eta_{q,2}(t) \left( \chi_{\ell_{q+1}} \star v_{q+1,2}\right) (x_1) & \quad \text{ if } t \in J_{q,2}, \\
    0 & \quad  \text{ if } t \in J_{q,3}.
\end{cases}
\end{equation}
Since $t_q \uparrow 1/2$, we have defined $u$ on $ [0,1/2) \times \T^2$. We extend it to $[1/2, 1] \times \T^2$ by
\begin{align}\label{eq:def_u}
    u(1/2, \cdot) \equiv 0, \qquad u (t,x) := - u(1-t, x) +  \sum_{q=0}^\infty \eta_{q,3} (1-t) v_{\text{swap}, q}(x) \quad\text{ for } t\in (1/2,1].
\end{align}

By construction, $u$ is H\"older continuous in space and smooth in time:

\begin{lemma} \label{lemma:regularity-holder}
The velocity field  defined in~\eqref{eq:def_u} satisfies
\begin{equation*}
    u \in C^{\infty}([0,1]; C^{\alpha}(\T^2))\quad \text{for any } 0< \alpha < \frac{1}{\gamma +(1-\gamma)(1+\delta)}.
\end{equation*}
In particular, for any $\alpha \in (0,1)$ and $\gamma \in (1/2,1)$ there exists $\delta \in (0,1)$ sufficiently small so that $u \in C^{\infty}([0,1]; C^{\alpha}(\T^2))$.
\end{lemma}

\begin{proof}
We need to show that $u \in C^{k}([0,1]; C^{\alpha}(\T^2))$ for all $k \in \N$.
Denote, for $t \in [t_q, t_{q+1}]$
$$
\tilde v_q(t, x) := \tilde v_{q,1}(t,x_2) + \tilde v_{q,2}(t,x_1) = \eta_{q,1}(t) \left( \chi_{\ell_{q+1}} \star v_{q+1,1}\right)(x_2) + \eta_{q,2}(t) \left( \chi_{\ell_{q+1}} \star v_{q+1,2}\right) (x_1) \,,
$$
and for $t \in [1-t_{q+1}, 1-t_{q}]$
$$
\tilde w_q(t, x) := -\tilde v_q(1-t, x) + \eta_{q,3} (1-t) v_{\text{swap}, q}(x).
$$
With these definitions,
$$
u(t, x) = \sum_{q=0}^{\infty} \tilde v_q(t, x) + \sum_{q=0}^{\infty} \tilde w_q(t, x)\quad \forall\,
t\in [0,1],\,x\in\T^2.
$$
In the following inequalities, just for this proof, we will write $\| \cdot\|_{L^\infty_t C^0_x}$ for the $L^{\infty}([t_q, t_{q+1}]; C^0_x)$-norm when dealing with $\tilde v_q$, as there will be no confusion; similarly for other norms involving time intervals.

From~\eqref{eq:norm-1},~\eqref{eq:norm-2} and~\eqref{eq:eta-time-mollifier} we have that
\begin{equation}\label{eq:estim_building_blocks}
\norm{\tilde v_q}_{L^{\infty}_t C^0_x} \lesssim q^2a_q;
\end{equation}
since $\tilde v_q$ is mollified with parameter $\ell_q$ we have
\[
\begin{aligned}
\norm{\nabla_x \tilde v_q}_{{L^{\infty}_t C^0_x}}  \lesssim q^2 \frac{a_q}{\ell_q}.
\end{aligned}
\]
By~\eqref{eq:eta-time-mollifier} and~\eqref{eq:t-q} we also have
\begin{align}
\norm{\partial_t^k \tilde v_q}_{{L^{\infty}_t C^0_x}} \leq  a_q \norm{\partial_t^k \eta_q}_{{L^{\infty}_t C^0_x}}
\lesssim a_q q^{2(k+1)}.
\end{align}
Arguing similarly, from~\eqref{eq:def_u_half} one sees that 
\begin{align}
\norm{\partial_t^k \nabla_x \tilde v_q}_{L^\infty_t C^0_x} \lesssim \frac{a_q^2}{\ell_q} q^{2(k+1)}.
\end{align}
For any smooth function $f:[a, b] \times \T^2 \to \R^2$, by interpolation estimates one has
\begin{align}
\norm{f}_{C^{k}_tC^{\alpha}_x}
\lesssim \norm{f}_{C^{k}_tC^{0}_x} + \norm{f}_{C^{k}_tC^{0}_x}^{1-\alpha} \norm{\nabla_x f}_{C^{k}_tC^{0}_x}^{\alpha}.
\end{align}
Using this and the facts that $a_q q^{2(k+1)}\lesssim_k a_q^{1/2} \leq 1$, $\ell_q \leq a_q$, we deduce that
\begin{align}
\norm{\tilde v_q}_{C^{k}_tC^{\alpha}_x}
& \lesssim a_q q^{2(k+1)} + \left( a_q q^{2(k+1)} \right)^{1-\alpha} \left( \frac{a_q^2}{\ell_q} q^{2(k+1)} \right)^{\alpha}\\
& \lesssim_k a_q^{1/2} + a_q^{1-\alpha}\left(\frac{a_q}{\ell_q}\right)^{\alpha}  \lesssim a_q^{1/2} + a_q^{1-\alpha\left( \gamma + (1-\gamma)(1+\delta) \right)}.
\end{align}
For the intervals of the type $[1-t_{q+1}, 1-t_q]$, all the previous estimates still hold, because the only difference between $\tilde v_q$ and $\tilde w_q$ is the swap term, which is spatially constant and satisfies
\[
\norm{\eta_{q,3} (1-t) v_{\text{swap}, q}}_{C^k_t C^0_x} \lesssim a_{q+1} q^{2(k+1)} \lesssim_k a_q^{1/2}.
\]
Hence
\begin{align} \label{eq:computation-holder-norm}
\norm{u}_{C^{k}_t C^\alpha_x}
\leq \sum\limits_{q \geq 0} \big(\norm{\tilde v_q}_{C^{k}_t C^\alpha_x} + \norm{\tilde w_q}_{C^{k}_t C^\alpha_x}\big)
\lesssim_k \sum\limits_{q \geq 0}  \big(a_q^{1/2} + a_q^{1-\alpha \left(\gamma+(1-\gamma)(1+ \delta) \right)}\big)
\end{align}
where the last quantity is finite if and only if $\alpha < [\gamma+(1-\gamma)(1+ \delta)]^{-1}$. In this case, as the series is absolutely convergent, we deduce that $u \in C^{k}([0, 1]; C^{\alpha}(\T^2))$ for any $k \geq 1$ as desired.

The last claim follows from the fact that, for any fixed $\gamma\in (1/2,1)$, one has
\begin{equation*}
    \lim_{\delta \downarrow 0} \frac{1}{\gamma+(1-\gamma)(1+ \delta)} = 1.\qedhere
\end{equation*}
\end{proof}

Henceforth we choose $\delta >0$ sufficiently small so that $u$ is $\alpha$-H\"older continuous, for given $\alpha\in (0,1)$ as in the statement of Theorem~\ref{th:main_theorem}.

\subsection{Properties of the Lagrangian flow}\label{subsec:properties_lagrangian_flow}

We need to analyze how chessboards are mapped by the flow associated to $u$. To this end, we define the horizontal lines where $v_{q, 1} $ is discontinuous by
\[
H_q:= \bigg\{ (x_1, x_2) \in \T^2: \lfloor x_2 \rfloor \in \frac{1}{a_q} \N\bigg\}.
\]
Similarly, the vertical lines where $v_{q,2}$ is discontinuous are given by
\[
V_q:= \bigg\{ (x_1, x_2) \in \T^2: \lfloor x_1 \rfloor \in \frac{1}{a_q} \N\bigg\}.
\]
The horizontal (resp. vertical) sets that are quantitatively distant from the horizontal (resp. vertical) lines above are given by
\[
\mathcal{H}_q:= \bigg\{x \in \T^2 : \text{dist}\left(x, H_q\right) > 2 \ell_q \bigg\},\quad
\mathcal{V}_q:= \bigg\{x \in \T^2 : \text{dist}\left(x, V_q\right) > 2\ell_q\bigg\}.
\]

The vector field $u:[0, 1/2) \times \T^2 \to \R^2$ defined in equation~\eqref{eq:def_u} satisfies the following properties. Let $1 \leq p \leq q$ be fixed, then:
\begin{equation}
\label{eq:prop_u_hor}
\left| x-y \right| \leq \ell_{q} \text{ and } y \in \mathcal{H}_{p} \implies u(t, x) = u(t, y) \quad \forall\, t \in J_{p-1,1},
\end{equation}
\begin{equation}
\label{eq:prop_u_ver}
\left| x-y \right| \leq \ell_{q} \text{ and } y \in \mathcal{V}_{p} \implies u(t, x) = u(t, y) \quad \forall\, t \in J_{p-1,2},
\end{equation}
\begin{equation}
\label{eq:prop_u_zero}
u(s, x) = u(s, y), \quad \forall\, t \in J_{p-1,3}\,, \quad \forall\, x, y \in \T^2.
\end{equation}

Let us prove~\eqref{eq:prop_u_hor}, the others being similar.
Note that, for any $ t \in J_{p-1,1}$, $u(t, \cdot )$ is obtained by a space mollification of a locally constant velocity field $\tilde v_{q,i}$, via a a convolution kernel of parameter $\ell_{p}$; therefore, $u(t, \cdot )$ is constant on each connected component of the set $\tilde{\mathcal{H}}_q:=\{ z : \text{dist}(z, H_{p}) > \ell_{p} \}$.
By the hypothesis and triangular inequality, one has
\[
\begin{aligned}
\text{dist}(x, H_{p}) \geq \text{dist}(y, H_{p}) - \big|x-y\big| > 2\ell_{p} - \ell_{q} \geq \ell_{p}
\end{aligned}
\]
so that $x\in \tilde{\mathcal{H}}_q$; since $|x-y|\leq \ell_q$, $x$ and $y$ belong to the same connected component of $\tilde{\mathcal{H}}_q$, giving the conclusion.

From now on, we will denote by $X^q_t$ the deterministic Lagrangian flow associated to 
\begin{equation}\label{eq:defn_uq}
    u_q (t,x) :=u (t,x) \cdot \one_{[t_q, 1-t_q]^c} (t);
\end{equation}
note that $u_q$ is a smooth velocity field, by the choice of the compactly supported time cut-offs~\eqref{eq:eta-time-mollifier}.
Since $u_q$ is divergence-free in space, the associated flow $X^q$ is measure preserving.

We define the set $A \subset \T^2$ in the statement of Theorem~\ref{th:main_theorem} as
\begin{align} \label{d:A}
    A := A^{(1)} 
\end{align}
where for any $j \in \N$ we define
\begin{equation} \label{d:A-j}
    A^{(j)} = A_1^{(j)} \cap A_2^{(j)} \cap A_3^{(j)} \cap A_4^{(j)},
\end{equation}
$$ $$
with 
\begin{align*}
    A_1^{(j)} & := \bigcap\limits_{q =j}^\infty \left(X_{t_{q-1}}^{q}\right)^{-1}(\mathcal{H}_{q}),\qquad\qquad\qquad \ \,
    A_2^{(j)} := \bigcap\limits_{q =j}^\infty \left(X_{t_{q-1} + (t_{q} - t_{q-1})/3}^{q}\right)^{-1}(\mathcal{V}_{q}),\\
    A_3^{(j)} & := \bigcap\limits_{q =j}^\infty \bigcap\limits_{1 \leq p  \leq q} \left(X_{1-t_{p}}^{q} \right)^{-1}(\mathcal{V}_{p}),\qquad\quad
    A_4^{(j)} :=  \bigcap\limits_{q =j}^\infty \bigcap\limits_{1\leq p  \leq q} \left(X_{1-t_p + 2(t_{p} - t_{p-1})/3}^{q} \right)^{-1}(\mathcal{H}_p).
\end{align*}

The Lebesgue measure of $A^{(j)}$ is indeed large for any $j \in \N$.

\begin{lemma}\label{lem:lebesgue_measure_A}
    Let $A^{(j)}$ and $A$ be defined by~\eqref{d:A} and \eqref{d:A-j}.
    Under Convention~\ref{convention_parameters}, there exists a constant $C>0$ independent of the choice of parameters such that
    \begin{equation*}
        |A^{(j)}|\geq 1-C a_j^{\delta (1- \gamma)} \quad \forall j\in\N.
    \end{equation*}
    In particular, for given $\eps >0$ as in Theorem~\ref{th:main_theorem}, we can choose $a_0 \in (0,1)$ small enough so that additionally
    $$ |A| \geq 1 - \eps.$$
\end{lemma}

\begin{proof}
    Clearly $|A^{(j)}| \geq 1 - \sum_{i=1}^4 | (A_i^{(j)})^c|$, so it suffices to prove the upper bound
    \begin{align*}
        |(A_i^{(j)})^c| \lesssim a_j^{\delta(1-\gamma)}\quad \forall\, j\in\N,\ i=1,\ldots,4.
    \end{align*}
Using the Lebesgue measure preserving property of $X^q$, we deduce that
$$ |(A_1^{(j)})^c| \leq \sum_{q=j}^\infty |\mathcal{H}_q^c| \leq 4 \sum_{q=j}^\infty a_q^{-1} \ell_q = 4 \sum_{q=j}^\infty a_q^{\delta(1-\gamma)} \leq 8 a_j^{\delta (1-\gamma)},$$
where the second to last inequality holds true thanks to~\eqref{eq:a-q}-\eqref{eq:ell-q}, and the last one thanks to~\eqref{eq:a-0}. The same estimate holds for $(A_2^{(j)})^c$. We now estimate $(A_3^{(j)})^c$, the argument for $(A_4^{(j)})^c$ being similar. It holds that
\[
\left( \bigcap\limits_{q \geq j} \bigcap\limits_{p  \leq q} \left(X_{1-t_p}^{q} \right)^{-1}(\mathcal{V}_p)\right)^c = \bigcup\limits_{q \geq j} \bigcup\limits_{p  \leq q} \left(X_{1-t_p}^{q} \right)^{-1}(\mathcal{V}_p^c) = \bigcup\limits_{p \geq j} \bigcup\limits_{q \geq p } \left(X_{1-t_p}^{q} \right)^{-1}(\mathcal{V}_p^c).
\]

If in the construction of the  vector field $u$ there were no swapping, we would have 
$$\bigcup\limits_{q \geq p } \left(X_{1-t_p}^{q} \right)^{-1}(\mathcal{V}_p^c) = \left(X_{1-t_p}^{p} \right)^{-1}(\mathcal{V}_p^c),$$ 
because of the symmetry $u(t,x )= - u(1-t, x)$.
However, because of the swapping of the vector field for $t \geq \frac12$, we need to define a restricted subset of $\mathcal V_p$ of size $10a_{p+1}$ (see the definition of $\tilde{\mathcal{V}}_p$ below) that provides an upper bound on the displacement of the trajectories induced by the velocity field $u_q$ on the time interval
\[
\mathcal I_p := \left[t_p + \frac{2(t_{p+1}-t_p)}{3},\, 1-t_p-\frac{2(t_{p+1}-t_p)}{3}\right],
\]
where the symmetry $u(t,x) = -u(1-t,x)$ breaks down.

We claim that, up to possibly choosing $a_0$ small enough, it holds
\begin{equation}\label{eq:claim_L1_smallness}
    \| u_q \|_{L^1 (\mathcal{I}_p ; L^\infty_x)} \leq 10 a_{p+1}.
\end{equation}
To see this, set $\tilde J_{q,i}:=1-J_{q,i}$; notice that $\mathcal I_p$ consists of the interval $J_{p,3}$ on which $u\equiv 0$, its simmetrization $\tilde J_{p,3}$ where $u$ acts in space as $v_{swap,p}$, and intervals $J_{q,i}$ or $\tilde J_{q,i}$ with $q\geq p+1$. Using the fact that $\| \eta_{q,i}\|_{L^1([0,1])}=1$ by~\eqref{eq:eta-time-mollifier}, as well as estimates~\eqref{eq:norm-1}-\eqref{eq:norm-2}, we obtain
\begin{align*}
    \| u_q \|_{L^1 (\mathcal{I}_p ; L^\infty_x)}
    & \leq \| u\|_{L^1(J_{p,3};L^\infty_x)} + \| u\|_{L^1(\tilde J_{p,3};L^\infty_x)} + \sum_{i\in {1,2,3}, q\geq p+1} (\| u\|_{L^1(J_{q,i};L^\infty_x)}+ \| u\|_{L^1(J_{q,i};L^\infty_x)})\\
    & \leq a_{p+1} + 2 \sum_{q\geq p+1} (a_q+ 2a_{q+1}) \leq 10 a_{p+1}
\end{align*}
where the last inequality follows from the superexponential decay of $(a_q)_q$, for $a_0>0$ small enough.
As a consequence of~\eqref{eq:claim_L1_smallness}, setting $\tilde{\mathcal V}_p := \{x: \text{dist} (x, \mathcal V_p^c)\geq  10 a_{p+1} \}$, one has $\tilde{\mathcal V}_p \subset \mathcal V_p$
and therefore
$$\bigcup\limits_{q \geq p } \left(X_{1-t_p}^{q} \right)^{-1}(\mathcal{V}_{p}^c) \subset \left(X_{1-t_p}^{p} \right)^{-1}( \tilde{\mathcal{ V}}_p^c) \,. $$
Therefore, using again \eqref{eq:a-q}, \eqref{eq:ell-q} and \eqref{eq:a-0}, we obtain that
\begin{align*}
    |(A_3^{(j)})^c|
    & \leq \sum\limits_{p \geq j} \left|\bigcup\limits_{q \geq p } \left(X_{1-t_p}^{q} \right)^{-1}(\mathcal{V}_p^c)\right|
    \leq  \sum\limits_{p \geq j} \left| \left(X_{1-t_{p}}^{p} \right)^{-1}(\tilde{\mathcal{V}}_p^c)\right|
    = \sum\limits_{p =j}^\infty \left|\tilde{\mathcal{V}}_p^c\right|\\
    & \leq 10 \sum_{p=j}^\infty a_p^{-1} (2\ell_p + a_{p+1})
    \leq 30 \sum_{p=j}^\infty a_p^{-1} \ell_p   \leq 60 a_j^{\delta(1-\gamma)} \,. \qedhere
\end{align*}
\end{proof}

We finally collect some fundamental properties satisfied by the Lagrangian flows $X_t^q$. For any $1 \leq p \leq q -1$, we have 
\begin{equation}
\label{eq:prop_X_1}
X_{t_{p-1} }^q (x) \in \mathcal{H}_{p} \implies X_{s}^q (x) \in \mathcal{H}_p \quad \forall\, s \in J_{p-1,1}\,,
\end{equation}
\begin{equation}
\label{eq:prop_X_2}
X_{s_{p-1}}^q (x)  \in \mathcal{V}_p \implies X_{s}^q(x) \in \mathcal{V}_p \quad \forall\, s \in J_{p-1,2} \,,
\end{equation}
where we recall \eqref{eq:t-q}, \eqref{eq:defn_J_qi}.

Recall the definition of $\mathcal{C}^{\mathrm{even}}_{a_0}$ (resp. $\mathcal{C}^{\mathrm{odd}}_{a_0}$), the even (resp. odd) chessboard of size $a_0$ given in~\eqref{eq:defn_chessboards}.
Let us define the restricted chessboards
\begin{equation}\label{eq:restricted:chessboards}
    F_e :=  \mathcal{C}^{\mathrm{even}}_{a_0} \cap \mathcal{V}_1 \cap \mathcal{H}_1, \quad
F_o =  \mathcal{C}^{\mathrm{odd}}_{a_0} \cap \mathcal{V}_1 \cap \mathcal{H}_1,
\end{equation}
see Figure~\ref{fig:placeholder}.
By their definition, we notice that
\begin{equation}
\label{eq:A_e_separated_A_o}
\text{dist}(F_e, F_o) \geq 2\ell_1.
\end{equation}
By  the definition of $u_q$ in~\eqref{eq:defn_uq} and that of $A$ in~\eqref{d:A}, it holds $X_1^q (A) \subset \mathcal{V}_1 \cap \mathcal{H}_1$ for every $q\in\N$; due to the use of the swap velocity field $v_{\text{swap},q}$ in~\eqref{eq:def_u}, it follows that 
\begin{equation}
    \label{eq:A_e}
\left\{\begin{aligned}
    X^{2q}_1(A \cap \mathcal{C}^{\mathrm{even}}_{a_0}) \subset F_e\,, \\
    X^{2q}_1(A \cap \mathcal{C}^{\mathrm{odd}}_{a_0}) \subset F_o\,,
\end{aligned}
\right.
\end{equation}
and
\begin{equation}
    \label{eq:A_o}
\left\{\begin{aligned}
    X^{2q+1}_1(A \cap \mathcal{C}^{\mathrm{even}}_{a_0}) \subset  F_o\,, \\
    X^{2q+1}_1(A\cap \mathcal{C}^{\mathrm{odd}}_{a_0}) \subset F_e\,.
\end{aligned}
\right.
\end{equation}

\begin{figure}
    \centering
    \includegraphics[width=0.85\linewidth]{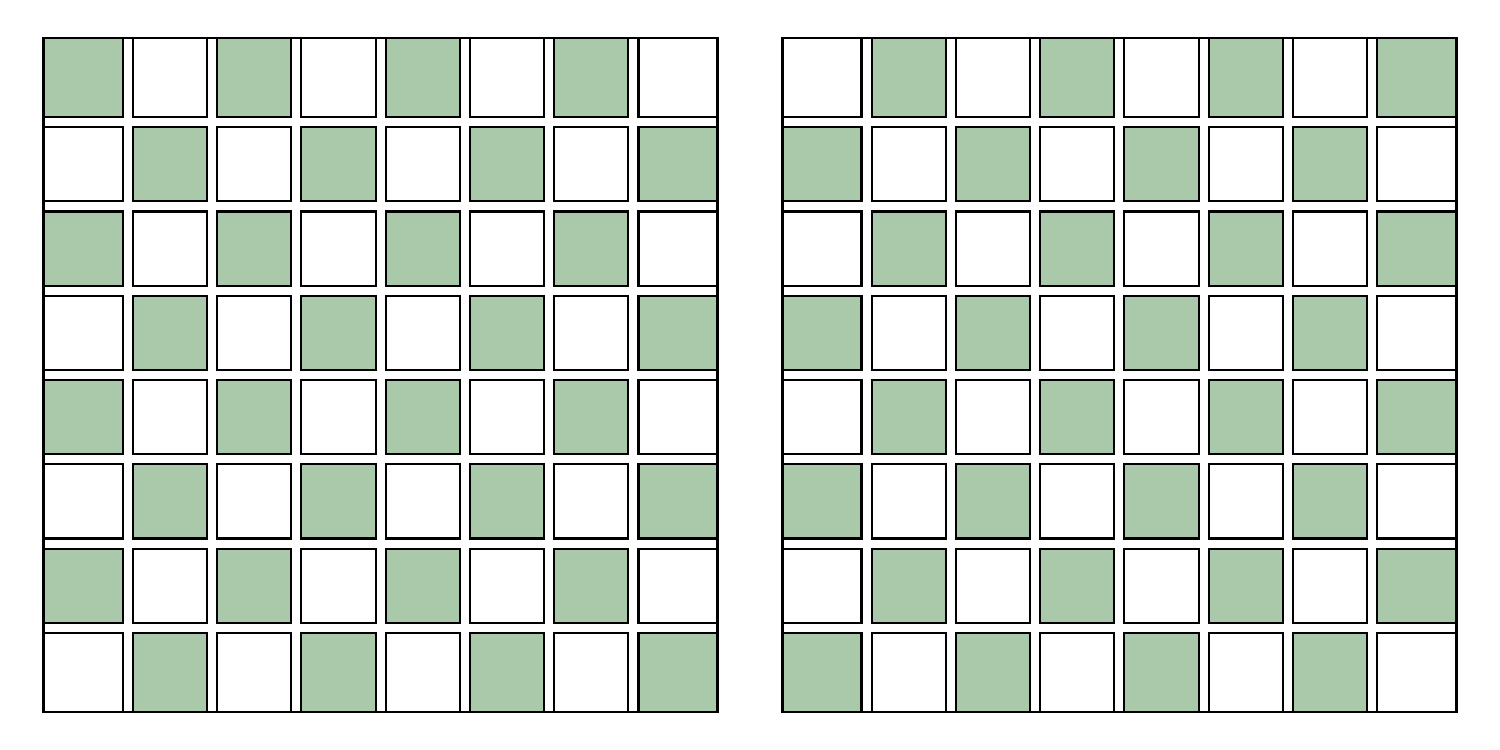}
    \caption{On the left, $F_e$ is shown in green. On the right, $F_o$ is shown in green.}
    \label{fig:placeholder}
\end{figure}

\section{Quantitative stability between flows under smooth advection}\label{sec:stability}

The proof of Theorem~\ref{th:main_theorem} relies on a quantitative comparison between the stochastic flow $X^{\kappa_q}$, associated with the velocity field $u$ and noise intensity $\kappa_q$, and the deterministic flow $X^q$, associated with $u_q$ as defined in~\eqref{eq:defn_uq}.

In this section, we establish some preliminary estimates in this direction.
They concern time intervals of the form $[0,t_q]$ (or their reflections) and perturbations with noise intensity $\kappa_q$; in this regime, we can regard the velocity field $u$ as being smooth, in the sense that it is active only at length scales much larger than $\kappa_q$ (equivalently, in Fourier space it is concentrated on frequencies $|k|\ll \kappa_q^{-1}$).

As the estimate in consideration will be essentially pathwise in nature, we state in a slightly more general framework. 
Let $\xi:[0,1] \to \R^2$ be a (bounded) càdlàg function, $y\in \T^2$ and let $X^{\kappa}(y)$ be {\em any} solution to the perturbed ODE\footnote{Since $u$ is continuous, is is possible to construct such a solution $X^\kappa(y)$ for any $y\in \T^2$ and any $\kappa\in\R$, by a Peano-type compactness argument; $X^\kappa(y)$ might not be unique, but the statement of Proposition~\ref{prop:up_to_tq} applies to any solution.}
\begin{equation}
\label{eq:pathwise-sde}
X^{\kappa}_t(y) = X^{\kappa}_0(y) + \int_0^t u(s, X^{\kappa}_s(y)) + \kappa \xi_t \quad\forall\, t\in [0,1].
\end{equation}
Let us denote from now on:
$$\norm{\xi}_{\infty}:=\sup\limits_{t \in [0, 1] } |\xi_t|.$$

\begin{proposition}
\label{prop:up_to_tq}
Let $a_0>0$ and $\delta>0$ be chosen sufficiently small and let $(\ell_q)_{q \in \N}$, $(\kappa_q)_{q \in \N}$, $(t_q)_{q \in \N}$ be defined as in Section~\ref{subsec:parameters}; let $u$ be the velocity field constructed in Section~\ref{subsec:construction_velocity} and $(u_q)_{q\in\N}$ be given by~\eqref{eq:defn_uq}. Let $A^{(j)}\subset \T^2$ be the set defined in~\eqref{d:A-j} for any $j \in \N$.

Then, for any $q\in \N$ there exists $\bar{q}= \bar{q} (q) \to \infty$ as $q \to \infty$, such that the following property holds true for any $q\in\N$. If $X^{\kappa_q}(y)$ is a solution to~\eqref{eq:pathwise-sde} and $\|\xi\|_\infty \leq q$, then for all $x\in A^{(\bar{q})}$ it holds that
\begin{equation}
\label{eq:prop_step1}
\sup_{t\in [0,t_q]} \left|X_t^{\kappa_q}(y) - X_t^q (x) \right| \leq 2 q^2 \kappa_q \norm{\xi}_{\infty} + \frac{\ell_q}{2q}\quad \text{for all $y$ such that}\quad |y-x| < q\kappa_q \norm{\xi}_{\infty} + \frac{\ell_q}{2q^2} \,,
\end{equation}
where $X_t^q$  denotes the Lagrangian flow  associated with the smooth vector field $u_q$.
\end{proposition}

\begin{proof}

For any $q \in \N$ we define $\bar{q} \in \N$
$$\bar{q} =\max \left \{ j\leq q :  \exp \left ( \int_0^{t_j} \| \nabla u (s)\|_{C^0_x} \dd s \right ) \leq q \right \}.$$
It is clear that $\bar{q} \to \infty$ as $q\to \infty$.

\textbf{STEP 0:} We have that $\bar{q} \leq q$, hence we can apply Gr\"onwall to $X_t^{\kappa_q} - X_t^q$ to get
$$ \sup_{t\in [0,t_{\bar q}]}\left|X_t^{\kappa_q}(y) - X_t^q (x) \right| \leq (|x-y|  + \kappa_q \| \xi \|_{\infty}) \exp \left ( \int_0^{t_{\bar q}} \| \nabla u (s)\|_{C^0_x} \dd s \right ) <  (q^2+q) \kappa_q \| \xi \|_{\infty} + \frac{\ell_q}{2 q} \,. $$

We now argue inductively; 
we will prove that for all $1\leq p \leq q - \bar{q}$, for all $x \in A^{(\bar q)}$, and for all $t \in [t_{\bar{q}}, t_{p+ \bar{q}}]$, it holds that

\begin{equation}\label{eq:up_to_tq_proof_claim}
\left|X_t^{\kappa_q}(y) - X_t^q (x) \right| < 3 p \kappa_q \norm{\xi}_{\infty} +(q^2+q) \kappa_q \| \xi \|_{\infty} + \frac{\ell_q}{2 q}\quad \text{for all $y$ s.t. } |y-x| <  q \kappa_q \, \norm{\xi}_{\infty} + \frac{\ell_q}{2 q^2} \,.
\end{equation}
Choosing $p=q-\bar q$ will then give the desired conclusion.

\textbf{STEP 1:} We treat the case $p=1$; therefore we may assume $\bar q<q$. Suppose by contradiction that the above claim does not hold; then there exist $x\in A^{(\bar q)}$ and $y\in\T^2$ with the properties that $|y-x| < (q^2+q) \kappa_q \| \xi \|_{\infty} + \frac{\ell_q}{2 q} $, and
\[
\tau := \inf \left\{s \in [t_{\bar q} , t_{\bar q +1}]: |X_s^{\kappa_q}(y) - X_s^q(x)| \geq 3 \kappa_q \| \xi \|_\infty + (q^2+q) \kappa_q \| \xi \|_{\infty} + \frac{\ell_q}{2 q}  \right\}
\]
is a well-defined element of $[t_{\bar q} , t_{\bar q +1}]$ (as the infimum is taken over a non-empty set).
Since $\xi$ is càdlàg, so is $X^{\kappa_q}(y)-X^q(x)$; therefore by the definition of $\tau$, one has
\begin{equation}\label{eq:up_to_tq_proof1}\begin{split}
    3 \kappa_q \| \xi \|_\infty + (q^2+q) \kappa_q \| \xi \|_{\infty} + \frac{\ell_q}{2 q} 
& \leq |X_{\tau}^{\kappa_q}(y) - X_\tau^q(x)| \\
& = \bigg| X_{t_{\bar{q}}}^{\kappa_q} (x)  - X_{t_{\bar{q}}}^{q} ( y) + \int_{t_{\bar{q}}}^{\tau} \bigg( u(s, X_{s}^{\kappa_q}(y)) - u_q(s, X_{s}^q(x)) \bigg) \dd s + \kappa_q \xi_{\tau} \bigg| \\
& \leq |X_{t_{\bar{q}}}^{\kappa_q} (x)  - X_{t_{\bar{q}}}^{q} ( y)|+\bigg| \int_{t_{\bar q}}^{\tau} \bigg( u(s, X_{s}^{\kappa_q}(y)) - u_q(s, X_{s}^q(x)) \bigg) \dd s\bigg| + \big|\kappa_q \xi_{\tau} \big| \\
& \leq |X_{t_{\bar{q}}}^{\kappa_q} (x)  - X_{t_{\bar{q}}}^{q} ( y)|+\int_{t_{\bar{q}}}^{\tau} | u (s, X_{s}^{\kappa_q}(y)) - u_q(s, X_{s}^q(x)) | \dd s + \kappa_q \norm{\xi}_{\infty}.
\end{split}\end{equation}

By definition of $A^{(\bar q)}$ and the fact that $\bar q<q$, one has $X_{t_{\bar q}}^q(x) \in \mathcal{H}_{\bar{q}+1}$; by~\eqref{eq:prop_X_1}, one then also has $X_s(x) \in \mathcal{H}_{\bar q}$ for $s \in J_{\bar q,1}$. Similarly, thanks to $x \in  A^{(\bar q)}$, we have 
\[
X_{t_{\bar q} + (t_{\bar q +1} -t_{\bar q})/3}^q(x) \in \mathcal{V}_{\bar q},
\]
as well as $X_s^q(x) \in \mathcal{V}_{\bar q}$ for $s \in J_{\bar q,2}$ by~\eqref{eq:prop_X_2}.
By the definition of $\tau$, for all $s \in [0, \tau)$ we have
\[
|X_{s}^{\kappa_q}(y)- X_{s}^q(x)| \leq  3 \kappa_q \| \xi \|_\infty + (q^2+q) \kappa_q \| \xi \|_{\infty} + \frac{\ell_q}{2 q}  \leq 2q^2 \kappa_q \norm{\xi}_{\infty} + \frac{\ell_q}{2} \leq  \ell_q,
\]
where the last inequality holds true because of the assumption $\| \xi\|_\infty\leq q$, which together with~\eqref{eq:kappa-q-ell-q} gives $2q^2\kappa_q \, \norm{\xi}_{\infty} \leq 2q^3 \kappa_q\leq \ell_q$. Recalling \eqref{eq:prop_u_hor}, \eqref{eq:prop_u_ver} and \eqref{eq:prop_u_zero}, we deduce that
\[
u(s, X_{s}^{\kappa_q}(y)) - u_q(s, X^q_{s}(x)) = 0\quad \forall\, s \in [t_{\bar q}, \tau).
\]

Inserting this fact in~\eqref{eq:up_to_tq_proof1} and using the inductive assumption~\eqref{eq:up_to_tq_proof_claim} with $p=0$, one concludes that
\[
   (q^2+q) \kappa_q \| \xi \|_{\infty} + \frac{\ell_q}{2 q} +  3 \kappa_q \| \xi \|_\infty  \leq |X_{t_{\bar{q}}}^{\kappa_q} (x)  - X_{t_{\bar{q}}}^{q} ( y)| + \kappa_q \|\xi\|_{\infty} < (q^2+q ) \kappa_q \| \xi \|_{\infty} + \frac{\ell_q}{2 q} + \kappa_q \| \xi \|_\infty \,, 
 \]
which yields a contradiction. Therefore~\eqref{eq:up_to_tq_proof_claim} must hold for $t\in [t_{\bar q},t_{\bar q +1 }]$.

\textbf{STEP 2:} We show that, if the property~\eqref{eq:up_to_tq_proof_claim} holds for $p-1$, then it holds for $p$, where $2 \leq p \leq q - \bar{q}$.
Let $x$, $y$ as before and assume by contradiction that there exists $t\in [t_{\bar{q}+ p-1},t_{\bar q + p}]$ such that~\eqref{eq:up_to_tq_proof_claim} does not hold; we can now consider
\[
\tau := \inf \left\{s \in [t_{\bar{q}+ p-1},t_{\bar q + p}]: |X_s^{\kappa_q}(y) - X_s^q(x)| \geq  3 p \kappa_q  \, \norm{\xi}_{\infty} +(q^2+q) \kappa_q \| \xi \|_{\infty} + \frac{\ell_q}{2 q} \right\}.
\]
By~\eqref{eq:pathwise-sde}, we have
\[
X_{t}^{\kappa_q} (y) = X_{t_{ \bar{q}+ p-1}}^{\kappa_q}(y) + \int_{t_{\bar{q} +p-1}}^t u (s, X_{s}^{\kappa_q}(y)) \dd s + \kappa_q  (\xi_t - \xi_{t_{\bar{q} + p-1}});
\]
subtracting a similar identity for the deterministic flow $X_t^q$, we have
\[
\begin{aligned}
X_t^{\kappa_q} (y) - X_t^q (x) = X_{t_{\bar{q} + p-1}}^{\kappa_q}(y) - X_{t_{ \bar{q} + p-1}}^q (x) + \int_{t_{\bar{q} + p-1}}^t  u (s, X_{s}^{\kappa_q}(y)) - u_q(s, X_{s}^q(x)) \dd s
 + \kappa_q  (\xi_t - \xi_{t_{ \bar{q} + p-1}}).
\end{aligned}
\]
By the inductive assumption
\[
\big|X_{t_{\bar{q}+ p-1}}^{\kappa_q}(y) - X_{t_{\bar{q}+ p-1}}^q (x) \big| <  3 (p-1) \kappa_q  \, \norm{\xi}_{\infty} +(q^2+q) \kappa_q \| \xi \|_{\infty} + \frac{\ell_q}{2 q}
\]
and by the definition of $\tau$ we have
\[
\big|X_{s}^{\kappa_q}(y) - X^q_{s}(x) \big| <  3 p \kappa_q  \|\xi\|_{\infty} +(q^2+q) \kappa_q \| \xi \|_{\infty} + \frac{\ell_q}{2 q}, \quad \text{ for all } s \in [t_{\bar{q}+ p-1}, \tau).
\]
Following similar arguments as in Step 1, one still has
\[
\int_{t_{p-1}}^{\tau} \bigg| u(s, X_{s}^{\kappa_q}(y)) - u_q(s, X^q_{s}(x)) \bigg| \dd s=0.
\]
Indeed, combining inequality~\eqref{eq:kappa-q-ell-q}, the fact that $p\leq q$ and the assumption $\|\xi\|_\infty\leq q$, we deduce that
\begin{align*}
\big|X_{s}^{\kappa_q}(y) - X_{s}^q(x) \big| \leq  3 p \kappa_q  \, \norm{\xi}_{\infty} +(q^2+q) \kappa_q \| \xi \|_{\infty} + \frac{\ell_q}{2 q} \leq 5 q^3 \kappa_q + \frac{\ell_q}{2 q} \leq   \ell_q\quad\forall\, s\in [t_{\bar{q}+ p-1},\tau);
\end{align*}
hence together with~\eqref{eq:prop_u_hor}, we deduce that $X_{s}^q(x) \in \mathcal{H}_{\bar{q}+ p}$ for all $s\in [t_{\bar{q}+ p-1}, \tau)$ and therefore
\[
u(s, X_{s}^{\kappa_q}(y, \omega)) - u_q(s, X_{s}^q(x)) = 0 \quad \forall\, s \in J_{\bar{q}+ p-1,1}\cap [0,\tau).
\]
Similarly, using~\eqref{eq:prop_u_ver} (resp.~\eqref{eq:prop_u_zero}), one concludes the same equality for all $s \in J_{\bar{q}+ p,2}\cap [0,\tau)$ (resp. $s\in J_{\bar{q}+ p,3}\cap [0,\tau)$).
Then, by the definition of $\tau$, the càdlàg property of $X^{\kappa_q}-X^q$, and the inductive assumption~\eqref{eq:up_to_tq_proof_claim}, we find that
\[
\begin{aligned}
& 3 p \kappa_q   \|\xi\|_{\infty}  +(q^2+q) \kappa_q \| \xi \|_{\infty} + \frac{\ell_q}{2 q}  \leq |X_{\tau}^{\kappa_q}(y, \omega) - X_\tau^q(x)| 
\\
&   \leq  \left | X_{t_{\bar{q}+ p-1}}^{\kappa_q}(y, \omega) - X_{t_{\bar{q}+ p-1}}^q(x, \omega) \right | 
 + \left | \int_{t_{\bar{q}+ p-1}}^{\tau} \bigg( u(s, X_{s}^{\kappa_q}(y, \omega)) - u_q(s, X^q_{s}(x)) \bigg) \dd s  \right | 
 + \left |\kappa_q (\xi_{\tau}-\xi_{t_{\bar{q}+ p-1}}) \right | \\
& < 3 (p-1) \kappa_q  \|\xi\|_{\infty}  +(q^2+q) \kappa_q \| \xi \|_{\infty} + \frac{\ell_q}{2 q}  + \underbrace{\int_{t_{p-1}}^{\tau} \bigg| u(s, X_{s}^{\kappa_q}(y, \omega)) - u_q(s, X^q_{s}(x)) \bigg| \dd s}_{=0} + 2 \kappa_q \norm{\xi}_{\infty}  \\
&= 3 (p-1) \kappa_q \|\xi\|_{\infty}  +(q^2+q) \kappa_q \| \xi \|_{\infty} + \frac{\ell_q}{2 q}  + 2 \kappa_q \norm{\xi}_{\infty},
\end{aligned}
\]
which implies
\[
3\kappa_q \norm{\xi}_{\infty} < 2\kappa_q \norm{\xi}_{\infty}.
\]
This is a contradiction, therefore concluding the proof.
\end{proof}

Similarly, one can prove stability between the Lagrangian and stochastic flows for $t \in [1-t_q,1]$ using the sets $A_3$ and $A_4$ in the definition of $A$ in~\eqref{d:A}. We omit the proof, as it is identical to the previous proposition.

\begin{proposition}
\label{prop:stability_after_singularity}
Consider the same setting as in Proposition~\ref{prop:up_to_tq} and suppose that $X^{\kappa_q}(y)$ is a solution to~\eqref{eq:pathwise-sde} with $\|\xi\|_\infty \leq q$.
Then, for all $x \in A^{(\bar{q})}$, it holds that
\begin{equation}
\sup_{t\in [1-t_q,1]} \left|X_{1-t_q, t}^{\kappa_q}(y) - X_{1-t_q, t}^{q}(x) \right| \leq 4 q^3 \kappa_q \| \xi\|_\infty + \frac{2}{3}\ell_q \quad \text{for all $y$ such that}\quad |x-y| \leq 2 q^2 \kappa_q \norm{\xi}_{\infty} + \frac{2 }{3q}\ell_q.
\end{equation}
\end{proposition}

\section{Proof of Theorem~\ref{th:main_theorem}}\label{sec:proof_main_thm}

Let $W$ satisfy Assumption~\ref{ass:noise} and $\alpha\in (\alpha_W,1)$, $\eps>0$ be given. We now choose $a_0$, $\delta>0$ sufficiently small so that the divergence-free velocity field $u$ constructed in Section~\ref{subsec:construction_velocity} satisfies $u\in C^\alpha([0,1]\times\T^2;\R^2)$ (by Lemma~\ref{lemma:regularity-holder}) and the set $A$ defined by~\eqref{d:A} satisfies $|A|\geq 1-\eps$ (by Lemma~\ref{lem:lebesgue_measure_A}).

We denote by $(X^\kappa_t(x,\omega))_{t \geq 0}$ the stochastic flow associated to the SDE
\[
X^\kappa_t = x + \int_0^t u(r,X^\kappa_r) \dd r + \kappa W_t,
\]
which exists thanks to condition $\alpha>\alpha_W$ and the results recalled in Section~\ref{sec:regbynoise_overview}.
Let us also recall that $X^q_t(x)$ denotes the deterministic flow associated to the approximation $u^q$ of $u$ defined by~\eqref{eq:defn_uq}. 

The proof of our main result relies on the next statement, showing that at time $t=1$ the stochastic flow $X^{\kappa_q}_1$ and the deterministic one $X^q_1$ get closer and closer as $q\to\infty$.

\begin{proposition}\label{prop:preliminary_main_thm}
   For a.e. $x\in \T^2$, it holds that
    \begin{equation}\label{eq:stability_preliminary_main_thm}
        \P\left( \omega\in\Omega: \lim_{q\to\infty} \dist ( X_1^{\kappa_q}( x, \omega), X^q_1( x))=0\right) =1.
    \end{equation}
\end{proposition}

\begin{proof}
Let $A^{(j)}\subset \T^2$ be the set defined in~\eqref{d:A-j} for any $j \in \N$. 
We recall that from the definition we have $A^{(j)}\subset A^{(j+1)}$ for any $j\in \mathbb{N}$ and $|A^{(j)}| \to 1$ as $j \to \infty$ from Lemma \ref{lem:lebesgue_measure_A}. Hence for a.e. $x\in \mathbb{T}^2$ there exists $\bar q \in \N$ so that $x\in A^{(\bar q)}$. We now fix such $x\in A^{(\bar q)}$.
Let us define 
\begin{equation}\label{eq:defn_Omega_q}
    \Omega_q := \left\{\omega\in\Omega: \sup\limits_{t \in [0, 1] } |W_t| \leq q\right\}.
\end{equation}
For all $\omega \in \Omega_q$, using $\xi_t:=W_t(\omega)$ in Proposition~\ref{prop:up_to_tq}, we have 
\begin{equation}\label{eq:main_prop_starting_point}
    \left|X_t^{\kappa_q}(x, \omega) - X^q_t(x) \right| \leq  2q^3\kappa_q +\frac{\ell_q}{2q} \quad  \forall t\in[0,t_q].
\end{equation}
Recall the definitions of $t_q$, $s_q$, $J_{q,i}$ from~\eqref{eq:defn_J_qi} and let us define
\begin{align}
\Omega^x_q := \left \{\omega\in\Omega:  \left| \int_{t_q}^{s_q} u_q(r, X_r^{\kappa_q}(x,\omega))\dd  r \right| \leq \frac{\ell_q}{20q} \right \} \cap \left \{ \omega\in\Omega:  \left| \int^{1- t_q}_{1- s_q} u_q(r , X_r^{\kappa_q}(x,\omega))\dd r \right| \leq \frac{\ell_q}{20q} \right \}.
\end{align}
Henceforth, to lighten the notation, we will often omit the dependence of stochastic processes on $(x,\omega)$.
Since $X^\kappa_q$ solves the SDE~\eqref{eq:intro_SDE}, we have
\begin{align}
X_{1-t_q}^{\kappa_q} - X_{t_q}^{\kappa_q} = \int_{t_q}^{1-t_q}u_q(r , X_r^{\kappa_q})\dd r + \kappa_q(W_{1-t_q} - W_{t_q});
\end{align}
therefore, if $\omega \in \Omega^x_q \cap \Omega_q$, then 
\[
\begin{aligned}
\left| X_{1-t_q}^{\kappa_q} - X_{t_q}^{\kappa_q} \right| &= \left| \int_{t_q}^{1-t_q}u_q(r , X_s^{\kappa_q})\dd r + \kappa_q(W_{1-t_q} - W_{t_q}) \right| \\
&\leq \left| \int_{t_q}^{s_q} u_q(r, X_r^{\kappa_q})\dd r \right| + \| u\|_{L^\infty ([s_q ,  1-s_q];L^\infty_x)} + \left| \int^{1- t_q}_{1-s_q} u_q(r, X_r^{\kappa_q})\dd r \right| + 2\kappa_q \sup\limits_{t \in [0, 1]}|W_t|
\\
& \leq \frac{\ell_q}{20q} + C (q+1)^2 a_{q+1} + \frac{\ell_q}{20q} + 2q\kappa_q
\end{aligned}
\]
where in the last passage we used the estimate $\| u\|_{L^\infty ([s_q ,  1-s_q];L^\infty_x)} \lesssim (q+1)^2 a_{q+1}$, which holds true thanks to definition of $s_q$ (cf.~\eqref{eq:defn_J_qi}), the construction of $u$ (cf.~\eqref{eq:def_u}) and estimate~\eqref{eq:estim_building_blocks}.
It then follows from the choice of the parameters~\eqref{eq:a-q}, \eqref{eq:ell-q} and~\eqref{eq:kappa-q} that, for all $q\geq q_0$ large enough, it must hold
\begin{align*}
    | X_{1-t_q}^{\kappa_q} - X_{t_q}^{\kappa_q}| \leq \frac{\ell_q}{6q}.
\end{align*}
Hence, for all $\omega \in \Omega_q^x \cap \Omega_q$ with $q$ sufficiently large, by the above estimate and~\eqref{eq:main_prop_starting_point} we have
$$
\begin{aligned}
|X^{\kappa_q}_{1-t_q}(x, \omega) - X^{q}_{1-t_q}(x)| &\leq |X^{\kappa_q}_{1-t_q}(x, \omega) - X^{\kappa_q}_{t_q}(x, \omega)| + |X^{\kappa_q}_{t_q}(x, \omega) - X^{q}_{t_q}(x)| + |X^{q}_{t_q}(x) - X^{q}_{1-t_q}(x)| \\
& \leq \frac{\ell_q}{6q} +  2q^3\kappa_q + \frac{\ell_q}{2q} +0 \,.
\end{aligned}
$$
By Proposition~\ref{prop:stability_after_singularity},  for all $t \geq 1-t_q$ and $\omega \in \Omega_q^x \cap  \Omega_q$, we deduce that
$$
\begin{aligned}
|X^{\kappa_q}_{t}(x, \omega) - X^{q}_{t}(x)| &= | X^{\kappa_q}_{1-t_q, t}(X_{1-t_q}^{\kappa_q}(x, \omega), \omega) - X^{q}_{1-t_q, t}(X_{1-t_q}^{q}(x)) |
\leq \frac{2}{3}\ell_q +  4 q^4\kappa_q.
\end{aligned}
$$
Overall, we have found that
\begin{align} \label{eq:estimate-stoch-det}
    |X^{\kappa_q}_{1}(x, \omega) - X^{q}_{1}(x)| \leq \frac{2}{3}\ell_q +  4 q^4\kappa_q \,, \quad \forall\, \omega \in \Omega^x_q \cap \Omega_q,\, q\geq q_0.
\end{align}

Let us consider
\begin{align} \label{d:tilde-omega}
        \tilde \Omega^x :=\liminf_{q\to\infty} (\Omega^x_q \cap \Omega_q)
        := \bigcup\limits_{p \geq 0}\bigcap\limits_{q \geq p} (\Omega^x_q \cap \Omega_q).
\end{align}
By definition, for any $ \omega \in \tilde \Omega^x$ there exists $p=p(\omega) \geq 0$ such that $\omega \in \Omega_q^x \cap \Omega_q$ for all $q \geq p$; in light of~\eqref{eq:estimate-stoch-det} and~\eqref{eq:kappa-q-ell-q}, it follows that
\begin{align}
\label{eq:lagrangian-estimate}
\limsup_{q\to\infty} |X^{\kappa_q}_{1}(x, \omega) - X^{q}_{1}(x)|
\leq \limsup_{q\to\infty} \left(\frac{2}{3}\ell_q +  4 q^4\kappa_q\right) = 0 \quad \forall\, \omega\in \Omega^x.
\end{align}

To deduce~\eqref{eq:stability_preliminary_main_thm}, we need to prove that $\P (\tilde \Omega^x) = 1$. We clearly have that $\Omega_{q} \subset \Omega_{q+1}$, hence it is straightforward to see that
\begin{equation*}
    \P \left( \bigcup\limits_{p \geq 0}\bigcap\limits_{q \geq p}  \Omega_q \right) =\P\left( \sup_{t \in [0, 1] } |W_t|<+\infty \right)=1;
\end{equation*}
to conclude, it then only remains to show that
\begin{align}
\label{eq:limsup-set-omega-tilde}
\P \left( \bigcup\limits_{p \geq 0}\bigcap\limits_{q \geq p} \Omega^x_q \right) = 1.
\end{align}

We now distinguish three cases depending on the stochastic process $W$, and we apply Lemma~\ref{lem:lucio_v2} with a fixed parameter $\theta \in (0,1]$, chosen as follows:
\begin{itemize}
    \item if $W$ is an fBm with parameter $H \in (0,1/2)$, we set $\theta = 1$;
    \item if $W$ is an fBm with parameter $H\in [1/2,+\infty)\setminus\N$, we set $\theta = \frac{1}{2H} - ( \frac{1}{4H} -1 + \gamma )$, which satisfies $\theta\in (0, \frac{1}{2H})$
    thanks to~\eqref{eq:ii};
    \item if $W$ is a $\beta$-L\'evy process with $\beta \in (0,2]$, we set $\theta = \frac{\beta}{2} - (\frac{\beta}{4} -1 + \gamma ) \in (0, \frac{\beta}{2})$, thanks to~\eqref{eq:iii}.
\end{itemize}

Let us write $X^{\kappa_q}_t=(X^{\kappa_q,(1)}_t,X^{\kappa_q,(2)}_t)$ to distinguish the coordinates of $X^{\kappa_q}_t$, similarly for $W^{(i)}_t$.
Recall by~\eqref{eq:def_u_half} that, on the intervals $J_{q,1}$ defined by~\eqref{eq:defn_J_qi}, the velocity field $u$ is defined is an horizontal shear flow: $u_q(r,x)=\tilde v_{q,1}(r,x_2)$.
Hence on such interval the SDE can be solved explicitly and one finds
$$
\int_{J_{q, 1}} u_q(r , X_{r}^{\kappa_q})\dd  r = \int_{J_{q, 1}} \tilde v_{q,1} \left(r, X_{t_q}^{\kappa_q,(2)} - \kappa_q W_{t_q}^{(2)} + \kappa_q W_{r}^{(2)}\right) \dd r.
$$
We can estimate the above quantity by applying Lemma~\ref{lem:lucio_v2} with
\begin{equation*}
    s=S=t_q,\quad t=T=s_q,\quad Z= X_{t_q}^{\kappa_q,(2)} - \kappa_q W_{t_q}^{(2)},\quad \psi\equiv 0,\quad f=\tilde v_{q,1}.
\end{equation*}
Therefore, by Markov's inequality and Lemma~\ref{lem:lucio_v2}, we deduce that
\begin{align*}
    \P \left( \omega \in \Omega:  \left| \int_{J_{q, 1}} u_q(r , X_{r}^{\kappa_q})\dd  r \right| > \frac{\ell_q}{20q} \right)
    & \lesssim \frac{q}{\ell_q} \left\| \int_{J_{q, 1}} u_q(r , X_r^{\kappa_q})\dd r \right\|_{L^2(\Omega)}\\
    & \lesssim q \frac{ \norm{\tilde v_{q,1}}_{L^\infty(J_{q, 1};C^{-\theta}_x)}}{\kappa_q^\theta \ell_q} \\
    & \lesssim \frac{ q^3 a_{q} a_{q+1}^\theta}{\kappa_q^\theta \ell_q}=C q^3 a_{q}^{\delta (\sfrac{\theta}{2} - 1 + \gamma)}, 
\end{align*}
where for the last inequality we used~\eqref{eq:norm-1}, while for the last equality we used~\eqref{eq:ell-q} and~\eqref{eq:kappa-q}.

By the choice of $\theta$, together with~\eqref{eq:i},~\eqref{eq:ii}, and~\eqref{eq:iii}, we have  $\sfrac{\theta}{2}-1+\gamma>0$, therefore
\begin{align*}
    \sum_{q \geq 1} \P \left( \omega \in \Omega:  \left| \int_{J_{q, 1}} u_q(r , X_{r}^{\kappa_q})\dd  r \right| >  \frac{\ell_q}{20q} \right) \lesssim \sum_{q\geq 1} q^3 a_{q}^{\delta (\sfrac{\theta}{2} - 1 + \gamma)} <+\infty
\end{align*}
which implies~\eqref{eq:limsup-set-omega-tilde} by the Borel-Cantelli lemma.
\end{proof}

Thanks to Proposition~\ref{prop:preliminary_main_thm}, we can transfer any non-selection property of flows $X^q_1$ to their stochastic counterparts $X^{\kappa_q}_1$.

\begin{proof}[Proof of Theorem~\ref{th:main_theorem}]
    For any $x \in A$, by~\eqref{eq:A_e_separated_A_o}, \eqref{eq:A_e} and \eqref{eq:A_o} it holds that 
$$ \liminf_{q\to\infty} \dist  (X_1^{2q} (x) , X_1^{2q+1} (x)) \geq 2 \ell_1 >0.$$
Combined with~\eqref{eq:stability_preliminary_main_thm} this implies that, for fixed $x \in A$, one has
$$
\P\left(\omega\in\Omega: \liminf_{q\to\infty}  \dist  (X_1^{\kappa_{2q}} (x, \omega) , X_1^{\kappa_{2q+1}} (x, \omega))  = \liminf_{q\to\infty} \dist  (X_1^{2q} (x) , X_1^{2q+1} (x))  \geq 2 \ell_1\right)=1.
$$
An application of Fubini-Tonelli's theorem then implies the conclusion~\eqref{eq:main_theorem}, for $A_\eps=A$ and $c_\eps=2\ell_1$.
To prove~\eqref{eq:main_theorem_eq2}, note that for fixed $x \in A$, by the compactness of $\T^2$, we can select a subsequence $(q_j)_j$ so that 
$$ X^{2q_j + 1 }_1 (x) \to y \neq y' \leftarrow X_1^{2 q_j} (x).$$
The conclusion then follows again by~\eqref{eq:stability_preliminary_main_thm}.
\end{proof}

We close this section with two byproducts of the construction from Theorem~\ref{th:main_theorem}. The first one asserts that, in the limit as $q\to\infty$, the random variables $X^{\kappa_q}_1(x)$ are becoming deterministic. Below $Var (Y)$ denotes the variance of a random variable $Y\in L^2(\Omega;\R)$.

\begin{corollary}\label{cor:vanishing_variance}
Let $u$  be given as in the proof of Theorem~\ref{th:main_theorem}. Then, for any $f\in C(\T^2)$ and a.e. $x \in \T^2$, one has
\[
\lim_{q \to \infty} Var(f(X_1^{\kappa_q}(x))) = 0.
\]
\end{corollary}

\begin{proof}
Since $\T^2$ is compact, $f$ is bounded and uniformly continuous, with modulus $h_f$. By the properties of the variance, one has
\begin{align*}
    Var(f(X_1^{\kappa_q}(x))) \leq \E[|f(X_1^{\kappa_q}(x))-f(X_1^q(x))|^2 ]\leq \E\Big[ h_f\big(\dist(X_1^{\kappa_q}(x),X_1^q(x))\big)^2\Big]
\end{align*}
By Proposition~\ref{prop:preliminary_main_thm}, for a.e. $x\in \T^2$, $\dist(X_1^{\kappa_q}(x),X_1^q(x))\big)$ converges $\P$-a.s. to $0$ as $q\to\infty$. The conclusion then follows by dominated convergence. 
\end{proof}

The next result states that, as $q\to\infty$, even the laws $Law(X^{\kappa_q}_1(x))$ do not admits a unique limit in $\mathcal{M}(\T^2)$. One way to capture this quantitatively is by means of the Wasserstein distance $\mathbb{W}_1$, which by Kantorovich-Rubinstein duality is given by
\begin{equation*}
    \mathbb{W}_1(\mu,\nu) = \sup\left\{\int_{\T^2} f(x)\mu({\rm d} x)-\int_{\T^2} f(x)\nu({\rm d} x) \Big\vert f:\T^2\to\R \text{ s.t. } |f(x)-f(y)|\leq \dist(x,y)\ \forall\, x,y\in\T^2\right\}.
\end{equation*}
Since $\T^2$ is compact, weak convergence of probability measures is equivalent to convergence w.r.t. $\mathbb{W}_1$.

\begin{corollary}\label{cor:lack_selection_laws}
    Let $u$ and $A$ be given as in the proof of Theorem~\ref{th:main_theorem}. Then there exists a constant $\tilde c_\eps>0$ such that, for all $x\in A$, one has
    \begin{equation*}
        \liminf_{q\to\infty} \mathbb{W}_1\big(Law(X^{\kappa_{2q}}_1(x)),Law(X^{\kappa_{2q+1}}_1(x))\big) \geq \tilde c_\eps >0.
    \end{equation*}
\end{corollary}

\begin{proof}
    Recall the sets $F_e$ and $F_o$ defined by~\eqref{eq:restricted:chessboards}; they are open sets for which~\eqref{eq:A_e_separated_A_o} holds, therefore we can find a Lipschitz function $f:\T^2\to\R$ such that $\| f\|_{Lip}\lesssim \ell_1^{-1}$ and $f\equiv 1$ on $F_e$, $f\equiv 0$ on $F_o$.
    It follows that
    \begin{align*}
        \mathbb{W}_1\big(Law(X^{\kappa_{2q}}_1(x)),Law(X^{\kappa_{2q+1}}_1(x))\big)
        & \geq \| f\|_{Lip}^{-1}\, |\E[f(X_1^{\kappa_{2q}}(x))-f(X^{\kappa_{2q+1}}_1(x))]|\\
        & \gtrsim \ell_1 \,|\E[f(X_1^{\kappa_{2q}}(x))-f(X^{\kappa_{2q+1}}_1(x))]|.
    \end{align*}
    Adding a subtracting $f(X^{2q}_1(x))-f(X^{2q+1}_1(x))$, applying Proposition~\ref{prop:preliminary_main_thm} together with dominated convergence, we obtain
    \begin{align*}
        \liminf_{q\to\infty} \mathbb{W}_1\big(Law(X^{\kappa_{2q}}_1(x)),Law(X^{\kappa_{2q+1}}_1(x))\big)
        & \gtrsim \ell_1 \liminf_{q\to\infty} |\E[f(X_1^{\kappa_{2q}}(x))-f(X^{\kappa_{2q+1}}_1(x))]|\\
        & =\ell_1 \liminf_{q\to\infty} |f(X_1^{2q}(x))-f(X^{2q+1}_1(x))| = \ell_1>0
    \end{align*}
    where the last equality follows from the properties~\eqref{eq:A_e}-\eqref{eq:A_o} and the construction of $f$.
\end{proof}

\section{Further consequences of Theorem~\ref{th:main_theorem}}\label{sec:consequences}
In this section we present some corollaries and variants of our construction. Specifically, in Section~\ref{subsec:density} we show that the velocity fields for which there is no unique selection in the vanishing noise limit are dense in suitable function spaces; Section~\ref{sec:vanishing_fractional_laplacian_PDE} is devoted to the proof of Corollary~\ref{cor:vanishing-viscosity-PDE}; finally, Section~\ref{subsec:autonomous} provides the construction of a time-independent velocity field $u$ for which selection does not occur in dimension $3$ or higher.

\subsection{Lack of selection for a dense set of velocity fields}\label{subsec:density}

Let $W$ be a stochastic process satisfying Assumption~\ref{ass:noise}. Recall that, for a given velocity field $u$, whenever it exists we denote by $(X^\kappa_t)_{t \geq 0}$ the unique stochastic flow associated to the SDE
\[
X^\kappa_t = x + \int_0^t u(r,X^\kappa_r) \dd r + \kappa W_t.
\]
We encode some of the key pathological properties satisfied by the velocity field constructed in Theorem~\ref{th:main_theorem} in the following definition.

\begin{definition}\label{defn:nonselection_property}
    Let $W$ satisfy Assumption~\ref{ass:noise} and let $u \in L^q([0,1]; C^\alpha(\T^d))$ with $(q,\alpha)$ satisfying the condition~\eqref{eq:condition_SWP_time_integrable} from Remark~\ref{rem:regbynoise_time_integrability} (so that $X^\kappa$ is well-defined for any $\kappa>0$); let $\eps\in (0,1)$.
    We say that $u$ has the \emph{$(1-\eps)$-non-selection property} if there exists $ A \subset \T^d$ with $| A| \geq 1 - \eps $ and a constant $c_\eps >0$ such that:
\begin{itemize}
    \item  there exists a sequence $\{ \kappa_q \}_{q}$ such that
    \begin{equation*}
        \P\left(  \omega\in \Omega: \liminf_{q\to\infty} \dist ( X_1^{\kappa_{2q}}( x, \omega), X_1^{\kappa_{2q+1}}( x, \omega))\geq c_\eps> 0\,, \text{ for a.e. } x \in A\right) =1\,,
    \end{equation*}
    \item for any $x \in  A$, there exist a subsequence $\{\kappa_{q'}\}_{q'}$ and points $y, y' \in \T^d$ such that
    $$ X_1^{\kappa_{2q'}} ( x,  \omega) \rightarrow y \neq y' \leftarrow X_1^{\kappa_{2q'+1}} ( x, \omega).\quad \text{for $\P$-a.e. } \omega\in\Omega.$$
\end{itemize}
\end{definition}

To show that the $(1-\eps)$-non-selection property is satisfied by many velocity fields, we will use the following simple rescaling idea. 
Let $[a,b]\subset [0,1]$ given and let $u$ be the velocity field 
defined in Section~\ref{sec:construction}.
Then we define ${u}^{a,b} : [a, b] \times \T^2 \to \R^2$ as:
\[
{u}^{a,b} (t, x) := \frac{1}{b-a}u\left (\frac{t-a}{b-a}, x \right ).
\]
Following similar computation as in Lemma~\ref{lemma:regularity-holder}, see~\eqref{eq:computation-holder-norm}, for any $\alpha < \frac{1}{\gamma+(1-\gamma)(1+ \delta)  }$, one can check that
$$\| u^{a,b} \|_{L^\infty_t C^\alpha_x} =  \frac{1}{b-a} \| u \|_{L^\infty_t C^\alpha_x} \lesssim \frac{1}{b-a} \sum\limits_{q \geq 0}  [a_q^{1/2} + a_q^{1-\left(\gamma+(1-\gamma)(1+ \delta) \right)\alpha}]  < \eps;$$
the last inequality holds true thanks to the super-exponential nature of $(a_q)_q$, up to choosing $a_0 \in (0,1)$ small enough, possibly depending on $b-a$. Hence, it is straightforward to conclude the following result.

\begin{corollary}
\label{cor:u-arbitrary-small}
Let $[a, b] \subset [0,1]$ be a given time interval. Then for any $\eta > 0$, $\alpha \in (0, 1)$ and $\eps>0$, there exists $v \in L^\infty([0, 1];C^{\alpha}(\T^2))$ satisfying the $(1-\eps)$-non-selection property, compactly supported on $[a,b] \times \T^2$ and such that
\[
\norm{v}_{L^\infty_t C^{\alpha}_x} \leq \eta.
\]
\end{corollary}

Since H\"older spaces $C^\alpha_x$ are not separable, let us introduce $\mathcal{C}^\alpha_x := \overline{C^\infty(\T^d)}^{C^\alpha_x}$, the completion of smooth functions with respect to the $C^\alpha_x$-norm.
Note that $\mathcal{C}^\alpha$ is separable and
\begin{align*}
    C^{\alpha'}\subsetneq \mathcal{C}^\alpha \subsetneq C^\alpha \quad\text{for any }0<\alpha'<\alpha<1.
\end{align*}
With these spaces at hand, we are now ready to state our density result (for simplicity, we only consider the case $d=2$, but the general one $d\geq 2$ is analogous).

\begin{corollary}
\label{cor:density-non-selection}
Let $W$ satisfy Assumption~\ref{ass:noise} and $\alpha_W$ be defined by~\eqref{eq:noise_regularity_threshold}; let $q\in [2,\infty)$  and $\alpha \in (\alpha_W,1)$.
Then the set of velocity fields which have the $(1-\eps)$-non-selection property (in the sense of Definition~\ref{defn:nonselection_property}) is dense in $L^q([0, 1];\mathcal{C}^{\alpha}(\T^2))$.
\end{corollary}

\begin{proof}
Let $v \in L^q([0, 1]; \mathcal{C}^\alpha(\T^d))$ and $\eta>0$ be given. Since $q< \infty$, the space $ L^q([0,1]; \mathcal{C}^\alpha(\T^d))$ is separable; in particular, for any $\eta >0$, one can find $v_\eta \in C^\infty ([0,1]\times \T^2) $ such that 
$$ \| v - v_\eta \|_{L^q_t C^\alpha_x} < \eta \qquad \text{and} \qquad  v_\eta \equiv 0 \quad \text{on } [1-\eta, 1] \times \T^2.$$
Let $\eps>0$ given as in the statement and set
\begin{equation}\label{eq:density_proof_eq2}
    C_\eta:= \exp (\| \nabla v_\eta \|_{L^1_t L^\infty_{x}} ), \qquad \tilde \eps:= \frac{\eps}{C_\eta}.
\end{equation}
Let $v$ be the velocity field given by Corollary~\ref{cor:u-arbitrary-small}, with $a= 1-\eta$, $b=1$ and $\eta$ as here, satisfying the $(1-\tilde \eps)$-non-selection property with associated set $\tilde A\subset \T^2$ and constant $ c_{\tilde \eps}$; define 
\begin{align*}
\overline{v}_\eta := \begin{cases}
v_\eta &\text{for }t \leq 1-\eta, \\
v & \text{for }t > 1-\eta.
\end{cases}
\end{align*}
By construction and triangular inequality, we have that
\begin{equation}\label{eq:density_proof_eq1}
    \|v - \bar v_{\eta} \|_{L^q_t C^\alpha_x}< 2\eta.
\end{equation}
Let $X_t(x)$ and $X_t^\kappa(x,\omega)$ denote respectively the Lagrangian and stochastic flows associated to $v_\eta$, which are well-defined thanks to its smoothness.

We claim that $\bar v_\eta$ has the $(1-\eps)$-non-selection property, with
\begin{equation*}
    A:= X_{1-\eta}^{-1}(\tilde A), \quad c_\eps:=c_{\tilde \eps};
\end{equation*}
combined with~\eqref{eq:density_proof_eq1}, this implies the desired conclusion.

First of all, since $v_\eta$ is smooth, it is quasi-incompressible with constant $C_\eta$ (cf.~\eqref{eq:density_proof_eq2}); by our choice of $\tilde \eps$, it holds that
\begin{align*}
    |A| = 1-|A^c|=1-|X_{1-\eta}^{-1}(\tilde A^c)| \geq 1 - C_\eta \tilde \eps = 1-\eps.
\end{align*}
Next, again thanks to the regularity of $v_\eta$, a pathwise application of Gr\"onwall's lemma yields
$$ \sup_{t\in [0,1-\eta]} |X_t (x) - X_t^{\kappa}(x,\omega)| \leq C_\eta \,\kappa\, \| W(\omega) \|_\infty.$$
We can now restart the dynamics at time $t=1-\eta$ and proceed as in the proof of Theorem~\ref{th:main_theorem}, working on the interval $[1-\eta,1]$ instead of $[0,1]$, where $\bar v_\eta=v$; the set $A$ was defined so that $X_{1-\eta}(x)\in \tilde A$ and $X^\kappa_{1-\eta}(x,\omega)$ gets arbitrarily close to it as $\kappa\ll 1$.
Specifically, applying Proposition~\ref{prop:up_to_tq}, Proposition~\ref{prop:stability_after_singularity}, and Lemma~\ref{lem:lucio_v2} as in the proof of Proposition~\ref{prop:preliminary_main_thm}, up to considering instead $\tilde\Omega_q = \{\omega : \|W(\omega)\|_\infty \leq (1+C_\eta)^{-1}q \}$, we obtain an analogue of estimate~\eqref{eq:stability_preliminary_main_thm}, for the same sequence $(\kappa_q)_q$ as therein. One can then go through the same argument as in Theorem~\ref{th:main_theorem}, which overall verifies the claim.
\end{proof}

\begin{remark}
    The same argument applies to other separable function spaces, for instance $L^q_t C^0_x$ or $L^q_t W^{s,p}_x$ with $s\in (0,1)$ and $p\in (1,\infty)$ (in both cases, under the assumption that $W$ is regularizing enough to guarantee strong well-posedness of the SDE at fixed $\kappa>0$); another example is $L^q_t \mathcal{C}^\alpha_{df}$, where $\mathcal{C}^\alpha_{df}$ denotes the collection of $f\in \mathcal{C}^\alpha_x$ such that $\nabla\cdot f=0$ in the sense of distributions.
\end{remark}

\begin{remark}
   For any $q\in [1,\infty)$, the set of velocity fields $u\in L^q([0,1];\mathcal{C}^\alpha_x)$ for which the ODE admits a unique, stable flow $X$ is residual in the sense of Baire, see the classical result by Orlicz~\cite{orlicz1932theorie} and more recently~\cite{cianfrocca2025some}; it's easy to check that such $u$ have the selection property, in the sense that $X_t(x)$ corresponds to the unique strong pointwise limit of the stochastic flows $X^\kappa_t(x)$, as $\kappa\to 0^+$.
   Therefore, one cannot expect to obtain any stronger statement than density in Corollary~\ref{cor:density-non-selection}; in particular, the set of velocity fields $u\in L^q([0,1];\mathcal{C}^\alpha_x)$ with the non-selection property is \emph{not} residual.
\end{remark}

\subsection{Non-selection of solutions to transport PDEs by vanishing fractional viscosity}
\label{sec:vanishing_fractional_laplacian_PDE}

We complete here the

\begin{proof}[Proof of Corollary~\ref{cor:vanishing-viscosity-PDE}]
Fix $\alpha\in (\alpha_W,1)$, $\eps>0$ and let $A$, $u$ be given as in Theorem~\ref{th:main_theorem}. In this case, it is convenient to work with the time-reversed velocity field
\begin{equation*}
    v(t,x):=-u(1-t,x)\quad\forall\, t\in [0,1],\, x\in\T^2.
\end{equation*}
Indeed, since $\beta$-stable Lèvy processes have a law which is invariant under time reversal, denoting by $X^\kappa$ the stochastic flow associated to $u$ and by $\tilde X^\kappa$ the one associated to $v$, one has
\begin{equation}\label{eq:time_reversed_flow_law}
    Law\big( (\tilde{X}^\kappa_{0,1})^{-1}(x)\big)=Law\big(X^\kappa_{0,1}(x))\quad\forall\,\kappa>0,\,x\in\T^2.
\end{equation}
Recall the sets $F_e$ and $F_o$ defined by~\eqref{eq:restricted:chessboards}; let $\rho_{in}$ be a smooth, non-negative function such that $\rho_{in}\equiv 1$ on $F_e$ and $\rho_{in}\equiv 0$ on $F_o$.
Let $\rho^\kappa$ denote the associated solutions to the PDE~\eqref{eq:PDE_fractional_laplacian}.

In light of the divergence-free property of $v$, the representation~\eqref{eq:FP_representation_formula2} and routine PDE arguments, one can deduce that:
\begin{itemize}
    \item $(\rho^\kappa)_{\kappa>0}$ is bounded in $L^\infty_{t,x}$ and $(\partial_t\rho^\kappa)_{\kappa>0}$ is bounded in $L^\infty([0,1];C^{-2}_x)$;
    \item The sequence $(\rho^\kappa)_{\kappa>0}$ is precompact in $L^\infty_{t,x}$, endowed with the weak-$\ast$ topology, as well as in $C([0,1];C^{-\gamma}_x)$, for any $\gamma>0$.
    \item Any limit point of $(\rho^\kappa)_{\kappa>0}$ is a weak solution $\rho\in L^\infty_{t,x}\cap C([0,1];C^{-\gamma}_x)$ to the transport equation~\eqref{eq:transport_PDE} with initial condition $\rho_{in}$. Moreover, one can extract a subsequence from $(\rho^\kappa)_{\kappa>0}$ with the additional property that $\rho^{\kappa_{q}}_t\rightharpoonup\rho_t$ weakly in $L^2_x$ for every $t\in [0,1]$.
\end{itemize}
We now apply the aforementioned facts to the sequence of couples $(\rho^{\kappa_{2q}},\rho^{\kappa_{2q+1}})$, therefore finding a (not relabelled for simplicity) subsequence such that
\begin{equation*}
    (\rho^{\kappa_{2q}},\rho^{\kappa_{2q+1}})\overset{\ast}{\rightharpoonup} (\rho,\bar\rho) \text{ in } (L^\infty_{t,x})^2, \quad (\rho^{\kappa_{2q}}_t,\rho^{\kappa_{2q+1}}_t)\rightharpoonup (\rho_t, \bar\rho_t) \text{ in } L^2_x \text{ for every }t\in [0,1].
\end{equation*}
In order to deduce that $\rho\neq \bar\rho$, it suffices to show the existence of $f\in L^2_x$ such that
\begin{equation*}
    \int_{\T^2} f(x) \big(\rho_1(x)-\bar\rho_1(x)\big) \dd x = \lim_{q\to\infty} \int_{\T^2} f(x) \big(\rho^{\kappa_{2q}}_1(x)-\bar\rho^{\kappa_{2q+1}}_1(x)\big) \dd x \neq 0.
\end{equation*}
To this end, take $f:=\mathbbm{1}_{A\cap \mathcal{C}^{\even}_{a_0}}$; combining the representation formula~\eqref{eq:FP_representation_formula2}, the equality of laws~\eqref{eq:time_reversed_flow_law}, Proposition~\ref{prop:preliminary_main_thm} and dominated convergence arguments, we see that
\begin{align*}
    \lim_{q\to\infty} \int_{\T^2} f(x) \rho^{\kappa_{2q}}_1(x) \dd x
    & = \lim_{q\to\infty} \int_{\T^2} f(x) \E\big[\rho_{in}   \big((\tilde X^{\kappa_{2q}}_{0,1})^{-1}(x)\big)\big] \dd x\\
    & = \lim_{q\to\infty} \int_{\T^2} f(x) \E\big[\rho_{in}   \big( X^{\kappa_{2q}}_1(x)\big)\big] \dd x\\
    & = \lim_{q\to\infty} \int_{\T^2} f(x) \E\big[\rho_{in}   \big( X^{2q}_1(x)\big)\big] \dd x
    = |A\cap \mathcal{C}^{even}_{a_0}|
\end{align*}
where in the last passage we used~\eqref{eq:A_e} and the choice of $\rho_{in}$. Arguing similarly and using~\eqref{eq:A_o} we see that
\begin{align*}
    \lim_{q\to\infty} \int_{\T^2} f(x) \rho^{\kappa_{2q+1}}_1(x) \dd x = 0
\end{align*}
which overall allows to conclude that $\rho\neq \bar\rho$.

Finally, lack of anomalous dissipation of energy follows from the Fluctuation-Dissipation Relation~\eqref{eq:fluctuation_dissipation}. Indeed, for any $q$ we have that
\begin{align*}
    \kappa_q ^\beta\int_0^t \big\| |\nabla|^\beta \rho^{\kappa_q}_s\big\|_{L^2}^2 \dd s = \frac12\int_{\T^d} Var\big[ \rho_{in}((X^{\kappa_q}_{0,t})^{-1}(x))\big] \dd x.
\end{align*}
By Corollary~\ref{cor:vanishing_variance}, the integrand converges to zero
for a.e. $x\in\mathbb T^2$, and by properties of the variance it is bounded by
$\|\rho_{in}\|_{L^\infty}^2$; therefore the conclusion follows by dominated convergence.
\end{proof}

\subsection{Autonomous velocity field for $d\geq 3$}\label{subsec:autonomous}

The goal of this section is to present an analogue of Theorem~\ref{th:main_theorem} with an autonomous velocity field $\bar u \in C^\alpha (\T^3;\R^3)$, by treating time as the third variable (the extension to $d\geq 4$ is then immediate).

To this end, we need a few preparations. In what follows, we will adopt the notation $\bar x= (x,\tilde x)$ with $\bar x\in \T^3$, $x\in \T^2$ and $\tilde x\in \T$ (the latter being the new component that takes the role of time); similarly, for stochastic processes we write $\bar X=(X,\tilde X)$, $\bar W=(W,\tilde W)$ and so on.

Given $\eps\in (0,1)$ and $\alpha\in (0,1)$, let $\eps_1$, $\eps_2>0$ be such that $(1-\eps_1)(1-\eps_2)\geq 1-\eps$; let $u:(1-\eps_2,1)\times\T^2\to \R^2$ be the compactly supported velocity field given by Corollary~\ref{cor:u-arbitrary-small}, for given $\eps_1,\alpha$ as in Theorem~\ref{th:main_theorem}. Let $A\subset \T^2$ be the set given by Theorem~\ref{th:main_theorem}, such that $|A|\geq 1-\eps_1$, and set
\begin{equation}\label{eq:defn_A_autonomous}
    \bar A := A\times [0,1-\eps_2]\subset \T^3.
\end{equation}
Let us extend $u$ first to $[0,1]\times\T^2$ by setting $u\equiv 0$ outside $(1-\eps_2,1)\times\T^2$, and then periodically to $\T^3$. Finally, define $\bar u\in C^\alpha(\T^3;\R^3)$ by
\begin{equation}\label{defn:autonomous_candidate}
    \bar{u}(x,\tilde x) := (u(\tilde x, x), 1).
\end{equation}
We claim that $\bar u$ is the desired autonomous velocity field; see Theorem~\ref{thm:autonomous_case} for the precise statement. To prove it, we need some additional non-trivial arguments, compared to the proof of Theorem~\ref{th:main_theorem}.

Let $\bar W=(W,\tilde W)$ be either a 3-dimensional fBm or $\beta$-stable process, so that $W$ and $\tilde W$ are independent (respectively 2- and 1-dimensional) processes of the same type.
We may assume $\bar W$ to be defined on a canonical space $(\bar\Omega,\bar\P)$ with $\bar\Omega=\Omega\times \tilde\Omega$, $\bar\P=\P\otimes\tilde\P$, where $\P$ and $\tilde \P$ denote respectively the laws of $W$ and $\tilde W$.
Let $\bar x\in\T^3$ and let $\bar X^\kappa=(X^\kappa,\tilde X^\kappa)$ be the unique solution to the SDE\footnote{As usual, we assume $\alpha\in (0,1)$ to be large enough in function of $H$ (resp. $\beta$) so that Theorem~\ref{thm:SWP_fBm} (resp. Theorem~\ref{thm:SWP_stable}) applies.}
\begin{equation}\label{eq:3d_SDE}
    \bar X^\kappa_t = \bar x + \int_0^t \bar u(\bar X^\kappa_s) \dd s + \kappa \bar W_t.
\end{equation}
By the definition~\eqref{defn:autonomous_candidate}, one has $\tilde X^\kappa_t=\tilde x + t +\kappa \tilde W_t$ thus yielding
\begin{equation}\label{eq:3d_reduced_SDE}
    X^\kappa_t
    = x + \int_0^t u(\tilde x + s+ \kappa \tilde W_s,X^\kappa_s) \dd s + \kappa W_t
    = x + \int_0^t u(\tilde x + s,X^\kappa_s) \dd s + \kappa Z^\kappa_t
\end{equation}
for
\begin{equation}\label{eq:3d_defn_psi}
    Z^\kappa_t:= W_t + \psi^\kappa_t, \quad \psi^\kappa_t:= \frac{1}{\kappa} \int_0^t \big[u(\tilde x + s+\kappa \tilde W_s, X^\kappa_s)-u(\tilde x + s, X^\kappa_s)\big] \dd s.
\end{equation}
The SDE~\eqref{eq:3d_reduced_SDE} has now the same structure of the original one considered in the main body of the paper, up to replacing $W$ with $Z^\kappa=W+\psi^\kappa$. The key point is to obtain sufficiently nice, uniform-in-$\kappa$ estimates, in order to show that $Z^\kappa$ has the same ``regularizing features'' as $W$, so to apply the stochastic estimates from Section~\ref{sec:stochastic_estimates}. This is the content of the next two lemmas.

Throughout this section, we take $(\cF_t)_{t\geq 0}$ to be the (standard augmentation of the) filtration given by $\cF_t:=\sigma((W_s)_{s\leq t}, (\tilde W_r)_{r\in[0,1]})$. Notice that $\tilde W$ is $\cF_0$-adapted, and therefore whenever conditioning on $\cF_t$, $\tilde W$ can be treated as a deterministic path.
Accordingly, we consider the quantities $\llbracket \cdot \rrbracket_\varsigma$ as defined in~\eqref{eq:slowly_varying_process}, but we allow them to be {\em random constants} depending on $\tilde W$; the $L^\infty(\Omega)$-norm appearing in~\eqref{eq:slowly_varying_process} must then be understood with respect to the $\mathbb{P}$, the law of $W$, and not the full law $\bar{\mathbb{P}}=\mathbb{P}\otimes \tilde{\mathbb{P}}$.

\begin{lemma}\label{lem:autonomous_apriori1}
    Let $\bar x\in\T^3$ and let $\bar X^\kappa$, $X^\kappa$, $Z^\kappa$ be given by~\eqref{eq:3d_SDE}-\eqref{eq:3d_reduced_SDE}-\eqref{eq:3d_defn_psi}. Then we have the pathwise estimate
    \begin{equation}\label{eq:pointwise_estimate_Z}
        \sup_{\kappa\in (0,1],t\in [0,1]} |Z^\kappa_t(\bar\omega)| \lesssim \sup_{t\in [0,1]} |\bar W_t(\bar\omega)| \quad\forall\,\bar\omega\in\bar\Omega
    \end{equation}
    where the hidden constant does not depend on $\bar x\in \T^3$. Moreover there exists a deterministic constant $C>0$ such that:
    \begin{enumerate}[label=(\roman*), ref=\roman*]
    \item If $\bar W$ is an fBm with $H\in (0,+\infty)\setminus \N$, then $\sup_{\kappa\in (0,1]} \llbracket X^\kappa \rrbracket_{H}\leq C$.
    \item If $\bar W$ is a $\beta$-stable process with $\beta\in (1,2)$, then $\sup_{\kappa\in (0,1]} \llbracket X^\kappa \rrbracket_{1/\beta}\leq C$.
\end{enumerate}
\end{lemma}

\begin{proof}
    By Lemma~\ref{lemma:regularity-holder}, $u$ is Lipschitz continuous in its first entry, therefore we have the pointwise-in-$\omega$ estimate
        \begin{align*}
            \sup_{t\in [0,1]} | \psi^\kappa_t|
            \leq  \frac{1}{\kappa} \int_0^1 |u(\tilde x + s+\kappa \tilde W_s, X^\kappa_s)-u(\tilde x + s, X^\kappa_s)| \dd s
            \leq \|\partial_t u\|_{C^0} \sup_{t\in [0,1]} |\tilde W_t|
    \end{align*}
    uniformly in $\kappa$; from this, estimate~\eqref{eq:pointwise_estimate_Z} readily follows.

    For the second part of the statement, similar arguments to those in the works~\cite{Gerencser2023,GalGer2025,BDG2025} show that, for fixed $\kappa>0$, one has $\llbracket X^\kappa \rrbracket_\varsigma<\infty$; so let us only show how to get a closed, uniform-in-$\kappa$ estimate for $\llbracket X^\kappa \rrbracket_\varsigma$ as a consequence of its finiteness.
    We only give the proof in the case of $\bar W$ being a $\beta$-stable process, the other one being similar (and to some extent simpler, since fBm admits moments of any order).
    By the properties of conditional expectation (cf.~\cite{Gerencser2023,GalGer2025}) and the definition of $(\cF_t)_t$, we have
    \begin{align*}
        \E_s|X^\kappa_t-\E_s X^\kappa_t|
        &\leq 2 \E_s\left|X^\kappa_t-X^\kappa_s-\int_s^t u(\tilde x + r+ \kappa \tilde W_r,\E_s X^\kappa_r) \dd r\right|\\
        & = 2 \E_s\left|\int_s^t [u(\tilde x + r+ \kappa \tilde W_r, X^\kappa_r) - u(\tilde x + r+ \kappa \tilde W_r,\E_s X^\kappa_r)] \dd r + \kappa (W_t-W_s)\right|\\
        & \lesssim \int_s^t \| u\|_{C^\alpha(\T^3)} \E_s\big[|X^\kappa_r-\E_s X^\kappa_r|^\alpha\big] \dd r + \kappa \E_s|W_t-W_s|\\
        & \lesssim \| u\|_{C^\alpha(\T^3)} \int_s^t (\E_s|X^\kappa_r-\E_s X^\kappa_r|)^\alpha \dd r + \E|W_t-W_s|\\
        & \lesssim \| u\|_{C^\alpha(\T^3)} \llbracket X^\kappa \rrbracket_{1/\beta}^\alpha |t-s|^{1+\frac{\alpha}{\beta}} + |t-s|^{1/\beta}.
    \end{align*}
    In the above, we used the facts that $W$ is independent of $\tilde W$ and thus an $(\cF_t)$-$\beta$ stable process, as well as condition $\beta>1$ to ensure finiteness of $\E|W_t-W_s|$ (cf.~\eqref{eq:moments_stable}); $\E|W_t-W_s|\sim |t-s|^{1/\beta}$ then follows by self-similarity of $W$.
    Taking the $L^\infty(\Omega)$-norm in the above estimate and dividing both sides by $|t-s|^{1/\beta}$, using that $\alpha>1-\beta$ and so $|t-s|^{1+\frac{\alpha-1}{\beta}}\leq 1$, one arrives at
    \begin{align*}
        \llbracket X^\kappa \rrbracket_{1/\beta}
        \leq C( \llbracket X^\kappa \rrbracket_{1/\beta}^\alpha + 1)
        \leq \alpha \llbracket X^\kappa \rrbracket_{1/\beta} + (1-\alpha) C^{\frac{1}{1-\alpha}}+C
    \end{align*}
    from which the conclusion follows after rearranging.
\end{proof}

\begin{remark}
    Note that, compared to Theorem~\ref{th:main_theorem}, Lemma~\ref{lem:autonomous_apriori1} requires the additional constraint $\beta>1$. As seen in the proof, this is needed to ensure that $\E[\| \bar W^\beta\|_\infty]<+\infty$, cf.~\eqref{eq:moments_stable}.
    Refining the arguments in the style of~\cite{BDG2025}, it might be possible to obtain useful a priori estimates in a larger range of parameters, plausibly up to $\beta>1/2$ (assuming that all the remaining steps below can still be appropriately readapted). We do not know however whether one should expect it to be possible to cover the whole range $\beta\in (0,2)$.
\end{remark}

\begin{lemma}\label{lem:autonomous_apriori2}
    Let $\bar x$, $\bar X^\kappa$, $X^\kappa$, $Z^\kappa$ be as in Lemma~\ref{lem:autonomous_apriori1}.
    Then it holds that
    \begin{equation}\label{eq:autonomous_apriori2}
        \sup_{\tilde x\in \T, \kappa\in (0,1]} \llbracket \psi^\kappa \rrbracket_\varsigma \lesssim \sup_{t\in [0,1]} |\tilde W_t|
    \end{equation}
    for the choice $\varsigma=1+\alpha H$ when $\bar W$ is an fBm with $H\in (0,+\infty)\setminus\N$ and $\varsigma=1+\frac{\alpha}{\beta}$ when $\bar W$ is a $\beta$-stable process with $\beta \in (1,2)$.
\end{lemma}

\begin{proof}
    Let us only consider the case when $\bar W$ is an fBm, the other one being similar. By Lemma~\ref{lemma:regularity-holder} and Taylor expansion, we have
    \begin{align*}
        \psi^\kappa_t
        = \int_0^t \tilde W_s \int_0^1 \partial_t u(\tilde x + s+\lambda \kappa \tilde W_s, X^\kappa_s) \dd \lambda \dd s.
    \end{align*}
    Let us set $u^{\tilde x}(t,x):=u(\tilde x + t,x)$ in the following.
    By the definition of $(\cF_t)_t$, properties of conditional expectation and the regularity $\partial_t u \in C^\alpha(\T^3)$, we find
    \begin{align*}
        \|\E_s |\psi^\kappa_t -\E_s \psi^\kappa_t|\|_{L^\infty(\Omega)}
        & \leq 2\left\|\E_s \Big|\psi^\kappa_t - \psi^\kappa_s-\int_s^t \tilde W_r \int_0^1 \partial_t u^{\tilde x}(r+\lambda \kappa \tilde W_r, \E_s X^\kappa_r) \dd \lambda \dd r\Big | \right\|_{L^\infty(\Omega)}\\
        &=2 \bigg\|\E_s \Big|\int_s^t \tilde W_r \int_0^1 \partial_t u^{\tilde x}(r+\lambda \kappa \tilde W_r, X^\kappa_r)-u^{\tilde x}(r+\lambda \kappa \tilde W_r, \E_s X^\kappa_r)  \dd \lambda \dd r\Big|\bigg\|_{L^\infty(\Omega)}\\
        & \lesssim \| \partial_t u^{\tilde x}\|_{C^\alpha(\T^3)} \sup_{t\in [0,1]} |\tilde W_t| \, \int_s^t \int_0^1 \|\E_s[ |X^\kappa_r-\E_s X^\kappa_r|^\alpha] \|_{L^\infty(\Omega)} \dd \lambda \dd r\\
        & \lesssim \| \partial_t u\|_{C^\alpha(\T^3)} \sup_{t\in [0,1]} |\tilde W_t| |t-s|^{1+\alpha H} \llbracket X^\kappa\rrbracket^\alpha_{H}.
    \end{align*}
    Dividing by $|t-s|^{1+\alpha H}$, taking supremum on both sides and applying Lemma~\ref{lem:autonomous_apriori1}, conclusion~\eqref{eq:autonomous_apriori2} follows.
\end{proof}

Recall the definition of the autonomous velocity field $\overline{u}: \T^3 \to \R^3$ given in~\eqref{defn:autonomous_candidate} and similarly consider 
$$\overline{u}_q (x, \tilde x):= (u_q (\tilde x, x), 1) \,.$$
We denote by  $\bar X^q_t(x)$ the associated deterministic flows of $\bar{u}_q$.
Let us also recall that  the velocity field $u$ has been rescaled so that by Corollary~\ref{cor:u-arbitrary-small} 
$$u \equiv 0 \quad \text{ on } (0,1 -\eps_2) \times \T^2 \qquad \supp (u) \subset (1 -\eps_2, 1) \times \T^2 \,. $$

 \begin{theorem}\label{thm:autonomous_case}
    For any $\alpha \in (0,1)$ and $\eps \in (0,1)$ there exist a divergence-free $u \in C^\alpha (\T^3)$, a set $\bar A \subset \T^3$ with $|\bar A| \geq 1 - \eps $, a subsequence $\{ \kappa_q \}_{q}$ and a constant $c_\eps >0$ such that the following holds true: for  every $\bar x\in\bar A$, it holds that
    \begin{equation*}
        \bar\P\left( \bar\omega\in\bar\Omega: \liminf_{q\to\infty} \dist (\bar X_1^{\kappa_{2q}}(\bar x,\bar \omega),\bar X_1^{\kappa_{2q+1}}(\bar x,\bar \omega))\geq c_\eps> 0\right) =1.
    \end{equation*}
    Furthermore, for any $x \in \bar A$ there exist a subsequence $\{\kappa_q\}_q$, points $y, y' \in \T^3$ such that for a.e. $\omega$ it holds that
    $$ \bar X_1^{\kappa_{2q}} (\bar x, \bar \omega) \rightarrow y \neq y' \leftarrow \bar X_1^{\kappa_{2q+1}} (\bar x, \bar\omega).$$
\end{theorem}

\begin{proof}
Let $\bar u$ given as above and $(\kappa_q)_q$ defined as in~\eqref{eq:kappa-q}.
We need to prove that the following holds for a.e. $x\in \T^3$:
\begin{equation}\label{eq:autonomous_case_stability_flows}
        \bar\P\left( \bar\omega\in\bar\Omega: \lim_{q\to\infty} \dist (\bar X_1^{\kappa_q}(\bar x,\bar \omega),\bar X^q_1(\bar x))=0  \right) =1.
    \end{equation}
    
    The proof follows similar arguments to those of Sections~\ref{sec:stability}-\ref{sec:proof_main_thm}, so we mostly sketch it, highlighting the main differences.
    Similarly to~\eqref{eq:3d_reduced_SDE}, the ODE satisfied by $\bar X^q$ can be reduced to a two-dimensional one with drift $(t,x)\mapsto u^q(\tilde x+t,x)$. Given that $\tilde x\in [0,1-\eps_2]$ and the support properties of $u$ and $u^q$, it is straightforward to see that $\bar X_1^q(x,\tilde x)=(X^q_1(x),\tilde x)$.
    The SDE~\eqref{eq:3d_reduced_SDE} has exactly the same structure of the one studied in Sections~\ref{sec:stability}-\ref{sec:proof_main_thm}, up to replacing $W$ with $Z^\kappa$ given in~\eqref{eq:3d_defn_psi}. Therefore,  the proofs of Propositions~\ref{prop:up_to_tq}-\ref{prop:stability_after_singularity} carry over verbatim, taking into account the pointwise estimate~\eqref{eq:pointwise_estimate_Z}.

Let $\tilde t_q := 1-\eps_2 - \tilde{x} + \eps_2 t_q$, for $t_q$ defined in~\eqref{eq:t-q}, to take into account the translation and rescaling of $u$; similarly, $\tilde J_{q,1}:=(1-\eps_2) - \tilde x+\eps_2 J_{q,1}$ is the rescaled version of the interval $J_{q,1}$ defined by~\eqref{eq:defn_J_qi}.
    In order to obtain the estimate~\eqref{eq:autonomous_case_stability_flows}, as in the proof of Proposition~\ref{prop:preliminary_main_thm},
    it remains to control the distance between $\bar X^{\kappa_q}_{1-\tilde t_q}(\bar x,\bar \omega)$ and $\bar X^{\kappa_q}_{\tilde t_q}(\bar x,\bar \omega)$ by pathwise estimates and Borel-Cantelli arguments; it then suffices to show that 
    \begin{equation}\label{eq:autonomous_proof_expectation_goal}
        \E_{\bar \P} \left[ \left| \int_{\tilde J_{q, 1}} u_q(r , X_r^{\kappa_q})\dd r \right| \right] \lesssim \kappa_q^{-\theta} \| u_q\|_{L^\infty(\tilde J_{q,1};C^{-\theta}_x)}.
    \end{equation}
    To obtain~\eqref{eq:autonomous_proof_expectation_goal}, it is convenient to first integrate w.r.t. $\P$ at $\tilde\omega$ fixed, and then integrate further in $\tilde \P$.
    Since $u_q$ is a shear flow on $\tilde J_{q,1}$, by~\eqref{eq:3d_reduced_SDE} we deduce that
    \begin{align*}
         (X_r^{\kappa_q})^{(2)}& =\left( X_{\tilde t_q}^{\kappa_q} -\kappa_q Z^{\kappa_q}_{\tilde t_q} + \kappa_q Z^{\kappa_q}_r \right)^{(2)},\\
         \int_{\tilde J_{q, 1}} u_q(r , X_{r}^{\kappa_q})\dd  r&  = \int_{\tilde J_{q, 1}} \tilde v_{q,1}\left(r, (X_{\tilde t_q}^{\kappa_q})^{(2)} - \kappa_q (Z^{\kappa_q}_{\tilde t_q})^{(2)} + \kappa_q (Z^{\kappa_q}_{r})^{(2)}\right) \dd r \,,
    \end{align*}
    where $\tilde v_{q,1}$ is defined in~\eqref{eq:def_u_half}. 
	We now apply Lemma~\ref{lem:lucio_v2} (cf. Remark~\ref{rem:practical_applications_SSL}) together with the estimate in Lemma~\ref{lem:autonomous_apriori2}  to deduce a pathwise estimate at $\tilde{\omega}$ fixed:
	$$\E \left[ \left| \int_{J_{q, 1}} u_q(r , X_r^{\kappa_q})\dd r \right| \right] \lesssim \kappa_q^{-\theta }  \sup_{t\in [0,1]} |\tilde W_t(\tilde \omega)| \| u_q\|_{L^\infty(\tilde J_{q,1};C^{-\theta}_x)}.$$
	Finally, taking expectation in $\tilde{\mathbb{E}}$ on both sides of the above estimate, we deduce~\eqref{eq:autonomous_proof_expectation_goal} and therefore also~\eqref{eq:autonomous_case_stability_flows}.
    With the latter estimate at hand, the rest of the argument follows exactly as in the proof of Theorem~\ref{th:main_theorem}.
\end{proof}

\appendix
\section{Smoothing estimates for stable Markov semigroups}\label{app:useful}

In the following, let $(P_t)_{t\geq 0}$ be the Markov semigroup associated to a $d$-dimensional Lévy process $L$ with characteristic exponent (symbol) $\Phi$, so that
\begin{align*}
    \E[e^{i \xi \cdot L_t}]= e^{-t \Phi(\xi)}\quad \text{and} \quad
    \widehat{P_t f}(\xi)=e^{-t \Phi(\xi)} \widehat{f}(\xi)
\end{align*}
whenever $f$ is regular enough. Notice in particular that derivatives $\partial_{x_i}$ and $P_t$ commute. 

\begin{lemma}\label{lem:stable_kernel_estimates}
    Let $(P_t)_{t\geq 0}$ be the Markov semigroup associated to a Lévy process $L$; assume that there exists $M>0$ such that
    \begin{equation}\label{eq:stable_kernel_assumption}
        \| D P_t f\|_{C^0(\R^d)} \leq M t^{-\frac{1}{\beta}} \| f\|_{C^0(\R^d)}\quad\forall\, t\in (0,1],\quad \forall\,f \in C^0(\R^d).
    \end{equation}
    Then for any $s_1,s_2\in\R$ with $s_1<s_2$ there exists a constant $C=C(s_1,s_2,M)$ such that
    \begin{equation}\label{eq:stable_kernel_conclusion}
        \|  P_t f\|_{B^{s_2}_{\infty,1}} \leq C t^{-\frac{s_2-s_1}{\beta}} \| f\|_{B^{s_1}_{\infty,\infty}}\quad\forall\, t\in (0,1], \quad \forall\,f \in B^{s_1}_{\infty,\infty}(\R^d).
    \end{equation}
\end{lemma}

\begin{proof}
    The statement is a refinement of~\cite[Prop. 3.3]{BDG2025}, so we shortly sketch the proof.

    Without loss of generality, we may assume $f$ to be a Schwartz function.
    Let $n\geq 2$; using the Markov semigroup property $P_t=P_{t/n}\circ \ldots \circ P_{t/n}$, the fact that $D$ and $P_t$ commute and estimate~\eqref{eq:stable_kernel_assumption}, one finds
    \begin{align*}
        \| D^n P_t f\|_{C^0(\R^d)}
        = \| (D P_{t/n})\circ \ldots (D P_{t/n}) f\|_{C^0(\R^d)} \leq M^n n^{n/\beta} t^{-\frac{n}{\beta}} \| f\|_{C^0(\R^d)}.
    \end{align*}
    Since $P_t$ is a Markov kernel, we always have $\| P_t f\|_{C^0(\R^d)} \leq \| f\|_{C^0(\R^d)}$; interpolating with the above estimate, one gets
    \begin{equation}\label{eq:intermediate_kernel_estim}
        \| P_t f\|_{C^s(\R^d)} \lesssim t^{-\frac{s}{\beta}} \| f\|_{C^0(\R^d)}.
    \end{equation}
    Next, using the definition of Littlewood--Paley blocks and applying Bernstein's inequalities and estimate~\eqref{eq:intermediate_kernel_estim} to each $\Delta_j f$, we have
    \begin{align*}
        \| P_t f\|_{B^{s_2}_{\infty,\infty}}
        & \sim \| \Delta_{-1} P_t f\|_{C^0(\R^d)} + \sup_{j\geq 0} 2^{j s_2}\| \Delta_j P_t f\|_{C^0(\R^d)}\\
        & \sim \| \Delta_{-1} P_t f\|_{C^0(\R^d)} + \sup_{j\geq 0} 2^{j s_1}\| P_t \Delta_j  f\|_{C^{s_2-s_1}(\R^d)}\\
        & \lesssim \| \Delta_{-1} f\|_{C^0(\R^d)} + t^{-\frac{s_2-s_1}{\beta}} \sup_{j\geq 0} 2^{j s_1}\| \Delta_j  f\|_{C^0(\R^d)}
        \lesssim t^{-\frac{s_2-s_1}{\beta}} \| f\|_{B^{s_1}_{\infty,\infty}}
    \end{align*}
    where the estimate is valid for any $s_2 \geq s_1$.
    The conclusion with $B^{s_2}_{\infty,\infty}$ replaced by $B^{s_2}_{\infty,1}$ finally follows by applying the interpolation estimate from~\cite[Thm. 2.80]{BCD2011}.
\end{proof}

\begin{remark}\label{rem:heat_kernel_estimates}
    Let $\mathcal{P}_t$ be the standard heat kernel on $\R^d$; it is well-known that it satisfies~\eqref{eq:stable_kernel_assumption} with $\beta=2$. Therefore in this case estimate~\eqref{eq:stable_kernel_conclusion} recovers the classical heat-kernel bound
    \begin{equation}\label{eq:heat_kernel_estimate}
        \|  \mathcal{P}_t f\|_{B^{s_2}_{\infty,1}} \leq C t^{-\frac{s_2-s_1}{2}} \| f\|_{B^{s_1}_{\infty,\infty}}\quad\forall\, t\in (0,1].
    \end{equation}
    In the case of $\beta$-stable processes, it's easy to check condition~\eqref{eq:stable_kernel_assumption} by a scaling argument; for a more general criterion based on the symbol $\Phi$, see~\cite[Thm. 1.3]{SSW2012}.
\end{remark}

\section*{Acknowledgments} 
    
    LG acknowledges support from the Istituto Nazionale di Alta Matematica (INdAM) through the project GNAMPA 2025 ``Modelli stocastici in Fluidodinamica e Turbolenza'' and the project GNAMPA 2026 ``Fluidodinamica stocastica: irregolarità, trasporto e fenomeni di regolarizzazione''. 
    FG acknowledges support from Imperial College London through the PhD Roth Scholarship and his PhD supervisor Dan Crisan for his continuous guidance. MS acknowledges support from the Chapman Fellowship at Imperial College London.
    
    This article is published with funding from Università degli Studi dell'Aquila under the Call for Proposals for Fundamental research and Early-career research grants - Year 2026. (Italian version: L'articolo è pubblicato con il contributo dell’Università degli Studi dell'Aquila nell'ambito dell'Avviso per la presentazione di Progetti di Ateneo per la Ricerca di base e Avvio alla Ricerca - anno 2026).

    The authors are grateful to Enrico Priola for useful discussions concerning existing results for SDEs driven by Lévy processes.

\bibliographystyle{plain}
\bibliography{myBiblio}

\end{document}